\documentclass[a4paper,11pt]{article}

\newcommand{\N}{\mathbb{N}}
\newcommand{\Z}{\mathbb{Z}}
\newcommand{\R}{\mathbb{R}}
\newcommand{\C}{\mathbb{C}}
\newcommand{\Q}{\mathbb{Q}}

\newcommand{\M}{\mathbb{M}}

\newcommand\norm[1]{\left\lVert#1\right\rVert}

\newcommand{\information}{{
  \bigskip
  \footnotesize
    
    \textbf{Jamerson Bezerra}:
    \textsc{
        Faculty of Mathematics and Computer Science, Nicolaus Co\-pernicus University, ul. Chopina 12/18, 87-100 Toruń, Poland.
    } \par\nopagebreak
    \textit{E-mail:} \texttt{jdbezerra@mat.umk.pl}

    \textbf{Sandoel Vieira}:
    \textsc{Universidade Federal do Piau\'i - UFPI, Rua Dirce Oliveira, 64048-550,  Ininga, Teresina, Brazil.} \par\nopagebreak
    \textit{E-mail:} \texttt{sandoel.vieira@ufpi.edu.br}
  
	\textbf{Carlos Gustavo Moreira}:
    \textsc{ Instituto de Matemática Pura e Aplicada - IMPA, Estrada Dona Castorina, 110, 22460-320, Rio de Janeiro, Brazil.} \par\nopagebreak
    \textit{E-mail:} \texttt{gugu@impa.br}
}}

\def\dim{\operatorname{dim}}

\def\diff{\operatorname{Diff}}
\def\per{\operatorname{Per}}

\def\max{\operatorname{max}}
\def\min{\operatorname{min}}

\def\spec{\operatorname{Spec}}

\def\diam{\operatorname{diam}}
\def\per{\operatorname{Per}}

\def\interior{\operatorname{Int}}

\def\tH{{\tilde{H}}}

\def\quand{\quad\text{and}\quad}

\def\GL{GL}

\usepackage[margin=1.2in]{geometry}

\usepackage{amsmath,amsthm,amssymb,amsfonts}
\usepackage{blindtext}
\usepackage{epsfig}
\usepackage{multicol}
\usepackage{xcolor}
\usepackage[hidelinks]{hyperref}
\usepackage{graphicx}
\usepackage{color}
\usepackage{psfrag}
\usepackage{multirow}
\usepackage{enumitem}
\usepackage{mathtools}%Allow text above \rightarrow with the command f.i. \xrightarrow[\text{world}]{\text{hello}}
\usepackage{dirtytalk}
\usepackage[all]{xy}
\usepackage{xfrac}

\usepackage{caption}
\usepackage{subcaption}

\newtheorem{theorem}{Theorem}

\newtheorem{proposition}[theorem]{Proposition}
\newtheorem{example}[theorem]{Example}

\newtheorem{remark}[]{Remark}

\numberwithin{theorem}{section}

\newtheorem{ltheorem}{Theorem}

\newcommand{\tLambda}{\tilde{\Lambda}}

\begin{document}

\title{Lagrange and Markov spectra for typical smooth systems}
\author{Jamerson Bezerra, Carlos Gustavo Moreira and Sandoel Vieira}

\maketitle

\begin{abstract}
    We prove that among the set of smooth diffeomorphisms there exists a $C^1$-open and dense subset of data such that either the Lagrange spectrum is finite and the dynamics is a Morse-Smale diffeomorphism or the Lagrange spectrum has positive Hausdorff dimension and the system has positive topological entropy.
    
    % We prove that among the set of smooth diffeomorphisms, in a compact manifold of dimension bigger or equal to two, admitting a transverse homoclinic intersection, there exists a $C^1$-open and $C^{\infty}$-dense subset of diffeomorphisms $\varphi$ such that, there exists a $C^1$-open and dense set of real functions $f$ of this manifold satisfying that the Lagrange and Markov spectra associated to the data $(\varphi,\, f)$ positive Hausdorff dimension. As a consequence, we obtain that there exists a $C^1$-open and dense subset of data $(\varphi,\, f)$ such that either the Lagrange spectrum is finite and $\varphi$ is a Morse-Smale diffeomorphism or the Lagrange spectrum has positive Hausdorff dimension and $h_{\text{top}}(\varphi)>0$.
\end{abstract}

%Content
\section{Introduction}

Dynamically defined Lagrange and Markov spectra are subsets of the real line that quantifies the asymptotic behavior of the orbits of a given system from the optics of a ``rule" provided by the level curves of a reference real function. More precisely, for a dynamical system $\varphi$ from a metric space $X$ to itself and a continuous function $f:X\to \R$ the \emph{Dynamically defined Lagrange and Markov spectra} associated to the \emph{data} $(\varphi,\, f)$ is given by the sets
\begin{align*}
    L(\varphi,\, f) = \left\{
        \limsup_{j\to\infty}\, f(\varphi^j(x))\colon\,
        x\in X
    \right\}
    \quad
    \text{and}
    \quad
    M(\varphi,\, f) = \left\{
        \sup_{j\in \Z}\, f(\varphi^j(x))\colon\,
        x\in X
    \right\}.
\end{align*}
In geometric terms, the Lagrange spectrum can be seen as the set of values (or heights) for which the respective level curves are asymptotically accumulated by orbits of the phase space. Similarly, the Markov spectra gathers the biggest heights from different orbits of the system.

A special case stands out in the number theoretical context in which the data $(\varphi,\, f)$ in consideration is given by the shift map $\varphi:\N^{\Z} \to \N^{\Z}$ and the real function $f:\N^{\Z}\to \R$ defined by $f((\theta)_{i\in\Z}) = [\theta_0; \theta_1, \theta_2, \ldots]+[0; \theta_{-1}, \theta_{-2}, \ldots]$
\footnote{
    The expression $t = [\theta_0; \theta_1,\ldots, ]$ represents the expansion of the real number $t$ in continuous fraction.
}.
The dynamically defined Lagrange and Markov spectra associated to this data coincides with the (classical) Lagrange and Markov spectra $M$ and $L$. These classical spectra originally appear from the analysis of Diophantine properties of real numbers and the study of its structure goes back to the nineteen century (see \cite{MA79}, \cite{MA80}) and since then many contributions have been made to the topic.

It is known that $M\cap (-\infty,3) = L\cap (-\infty,3)$ is a discrete set with $3$ as the sole accumulation point (see \cite{MA79, MA80}). Moreover, there exists an (optimal) constant $c_F>3$ with the property that $[c_F,\infty)\subset L\subset M$ (see \cite{FR75,H49}). The fractal properties of these spectra between $3$ and $c_F$ is theme of active research from the last 60 years. Recently, Moreira, in \cite{M02018}, showed that the sets $L\cap (-\infty, t)$ and $M\cap (-\infty, t)$ cannot be differentiated using the Hausdorff dimension. Furthermore, at any small interval around $3$ these spectra have positive Hausdorff dimension. This last result was improved in \cite{HaMoGuRo2022} where the precise modulus of continuity of the function $\varepsilon\mapsto \dim_H(L\cap (3,\, 3+\varepsilon)) = \dim_H(M\cap (3,\, 3 + \varepsilon))$ was analyzed.

The analysis of the set of points that belongs to the classical Markov but are not in the Lagrange spectrum has a central place in the theory. Indeed, $M\backslash L$ has a complex fractal structure and the precise estimation of its Hausdorff dimension is still far from being established (see \cite{BoLMMR} for a comprehensive discussion on the structure of $M\backslash L$). Regarding the study of the set of real numbers possessing the same Lagrange value or equivalently the study of the ``level curves of the classical Lagrange spectrum", Moreira and Villamil, in \cite{MoVi2023}, proved that for each given value $t$ in the interior of $L$ the Hausdorff dimension of the real numbers whose Lagrange value belongs to $(-\infty, t)$ coincides with the Hausdorff dimension of the level set associated to $t$.

Different characterizations of the classical spectra appear in the literature aiming to introduce new ideas and techniques into the subject. For instance, it is known that the Lagrange spectra can be realized using the geodesic flow on the Modular surface (see \cite{Ar1994}). The dynamical characterization presented in this work is due to Perron \cite{PERRON} and it provides a different view for many of the classical arguments in the theory. Furthermore, it allows us to explore the properties of these sets for more general dynamical contexts such as the one of smooth dynamics. In this work we address following question:

% One of such comes from the context of hyperbolic dynamics or more specifically dynamics admitting horseshoes. Through the use of a coding describing the orbits of these systems many of the combinatorial techniques from the classical analysis of the spectra can still be applied.

\begin{center}
    \it
    For typical smooth data, is the complexity of the dynamically defined Lagrange and Markov spectra enough indication of the complexity of the system? 
\end{center}

That is indeed the case when dealing with the Lagrange spectrum which typically contains enough information to capture the complexity of the system in analysis as described by the following dichotomy.
 
% The answer is positive if we use the Lagrange spectrum for reference as indicated by the following dichotomy. 
\begin{ltheorem}
\label{thm:dichotomy}
    Let $M$ be a compact manifold. Then, there exists a $C^1$-open and $C^1$-dense subset $\mathcal{R}(M)\subset \diff^{\infty}(M)\times C^1(M;\R)$ such that if $(\varphi,\, f)\in \mathcal{R}(M)$, then either
    \begin{enumerate}
        \item $L(\varphi,\, f)$ is finite and $\varphi$ is a Morse-Smale 
        \footnote{
            A diffeomorphism is \emph{Morse-Smale} if the chain-recurrent set is hyperbolic and finite.
        }
        diffeomorphism or;
        
        \item $L(\varphi,\, f)$ has positive Hausdorff dimension and $h_{\text{top}}(\varphi)>0$.
    \end{enumerate}
\end{ltheorem}
Therefore, at least generically, if the Hausdorff dimension of the Lagrange spectrum is positive (complexity of the spectrum) we should expect our system to have positive entropy (complexity of the system). Such a result is not possible for the Markov spectrum though. Indeed, in any compact manifold, it is possible to build examples of open sets of smooth data where the dynamics is very predictable, nevertheless and the Markov spectrum contains intervals.

\begin{ltheorem}
\label{thm:bigMarkov}
    Let $M$ be a compact manifold. Then, there exists an open subset $\mathcal{U}\subset \diff^{\infty}(M)\times C^1(M;\R)$ such that for every $(\varphi, f) \in \mathcal{U}$, we have that $L(\varphi,\, f)$ is finite and $M(\varphi,\, f)$ has non-empty interior.
\end{ltheorem}

Hidden is the statement of Theorem \ref{thm:dichotomy} lies the fact that a typical diffeomorphism in the $C^1$-topology is either a Morse-Smale diffeomorphism or admits a horseshoe (see \cite{Cr2010}).  So, in this work we analyze Lagrange spectrum for typical data in which the systems admits a horseshoe or equivalently a transversal homoclinic intersection associated to a hyperbolic periodic point of saddle-type.

For diffeomorphisms admitting a horseshoe most of the results in the literature about the dynamically defined spectra relies not only on the symbolic representation of the dynamics but also the geometric properties of the invariant set itself. That is the case, for instance, when the horseshoe lies in a two dimensional environment. In fact, geometric properties such as regularity of the invariant distributions, regularity of the holonomy maps and the dependency of these objects with respect to the diffeomorphism on surfaces have been explored since the 60's (see \cite{PaTa1995} for a comprehensive description of those geometric properties in the surface case).

Exploring these properties Moreira and Roman\~a in \cite{IBMO2017} proved that both spectra have non-empty interior for typical diffeomorphisms with a ``thick" surface horseshoe $\Lambda$ ($\dim_H(\Lambda) > 1$) for typical real function. Dealing with ``thin" surface horseshoes $\Lambda$ instead ($\dim_H(\Lambda) < 1$), Cerqueira, Matheus and Moreira in \cite{MMC} approached the problem of continuity of the maps $t\mapsto \dim_H(L(\varphi|_{\Lambda}, f)\cap (-\infty,\, t))$ and $t\mapsto \dim_H(M(\varphi|_{\Lambda}, f)\cap (-\infty,\, t))$. They guaranteed that continuity holds for typical smooth systems $\varphi$ and typical real function $f$. If the systems in analysis preserve area, they obtained, additionally, that  typically $\dim_H(L(\varphi|_{\Lambda}, f)\cap (-\infty,\, t)) = \dim_H(M(\varphi|_{\Lambda}, f)\cap (-\infty,\, t))$ (see also \cite{CeMoRo2022} for a similar result in the context of geodesic flow on negatively curved surfaces). The thin assumption on the horseshoe was later removed by Lima, Moreira and Villamil in \cite{LiMoVi2023}.

Since we can recover the classical Lagrange and Markov spectra from this smooth setting (\cite{Ar1994} and \cite{LiMo2022}), the study of the dynamically defined spectra for typical smooth system can provide a way to infer properties for the classical $M$ and $L$ that we can see for typical data in the smooth counterpart. One example of that is provided by the phase transition property of the dynamically defined Markov Lagrange spectra for typical conservative dynamics of horseshoes obtained by Lima and Moreira in \cite{LiMo2021} in which there is a threshold  $t^*$ where the portion of both spectra inside of the interval $(-\infty,\, t - \delta)$, for any $\delta>0$, has Hausdorff dimension smaller than one, however, just after $t^*$, we can already see non-empty interior.

Nothing much is known once we leave the surface setting. One of the main reasons is that for typical horseshoes in higher dimensions we no longer have the nice geometrical properties observed in the two dimensional case. Nevertheless, it is possible to perform local constructions to design, after perturbation, such hyperbolic sets presenting the desired geometric features. This is the type of technique used by Palis and Viana in \cite{PaVi1994} to investigate the abundance of diffeomorphism displaying infinitely many coexistent sinks for dissipative systems.

Combining the construction in \cite{PaVi1994} with an adaptation of the techniques developed in \cite{IBMO2017}, in this work we prove the following result.
\begin{ltheorem}
\label{thm:Main}
    Let $M$ be a compact manifold of dimension $d\geq 2$. Let $\varphi\in \diff^{\infty}(M)$, admitting a transversal homoclinic intersection. Then, there exist $\varphi'\in \diff^{\infty}(M)$, $C^{\infty}$-close to $\varphi$ and a $C^1$-open neighbourhood of $\varphi'$, $\mathcal{U}(\varphi')\subset \diff^{\infty}(M)$ such that for every $\tilde{\varphi}\in \mathcal{U}(\varphi')$ there exists an open and dense set, $\mathcal{X}_{\tilde{\varphi}}\subset C^1(M;\mathbb{R})$, of real functions such that if $f\in \mathcal{X}_{\tilde{\varphi}}$, then both $M(\Lambda_{\tilde{\varphi}},\, f)$ and $L(\Lambda_{\tilde{\varphi}},\, \varphi)$ have positive Hausdorff dimension.
\end{ltheorem}

In the Section \ref{Pre}, we introduce the notation and establish the preliminary results that will be used a throughout the work. In Section \ref{section:lagrangeSymbolic}, we analyze the dynamically defined Lagrange spectrum associated to subshifts of finite type. The choice of the generic set of pairs that will be used in the proof of Theorem \ref{thm:Main} is provided in Section \ref{section:choiceData}. Section \ref{section:proofResults} contains the proof of the theorems \ref{thm:dichotomy}, \ref{thm:bigMarkov} and \ref{thm:Main}.

\paragraph{Acknowledgements:} Research was partially supported by the Narodowe Centrum Nauki Grant 2022/45/B/ST1/00179, by the Center of Excellence ``Dynamics, Mathematical Analysis and Artificial Intelligence" at Nicolaus Copernicus University in Toruń and by FCT-Funda\c{c}\~{a}o para a Ci\^{e}ncia e a Tecnologia through the project  PTDC/MAT-PUR/29126/2017. We would like to thank Davi Lima and Sergio Romaña for their careful reading and helpful suggestions in the early stages of the manuscript.
\section{Preliminaries}
\label{Pre}

In this section we introduce the basic tools that will be used in the work.

%%%%%%%%%%%%%%%%%%%%%%%%%%%%%%%%%%%%%%%%%%%%%%%%%%%%%
\subsection{Dynamically defined Markov and Lagrange spectra}

Let $X$ be a metric space and $\varphi:X\to X$ be a homeomorphism. Let $\Lambda$ be a $\varphi$-invariant compact subset of $X$ and $f:X\to \R$ a continuous function. The  \emph{dynamically defined Lagrange and Markov spectra} over $\Lambda$ is given respectively by the sets $L(\Lambda,\, f):= L(\varphi|_{\Lambda},\, f)$ and $M(\Lambda,\, f) := M(\varphi|_{\Lambda},\, f)$.

It is not hard to see that $L(\Lambda,\, f)\subset M(\Lambda,\, f)\subset f(\Lambda)$. Notice that we cannot expect in general that these spectra capture a good dynamical behaviour of our system. Indeed, we could always consider $f$ a constant function and in this case the spectra is trivial.  But, triviality of these spectra also occur for a big class of systems independently of the chosen real function. That is the case when the limit set of the dynamics is finite and so the Markov and Lagrange are finite and coincide. This is exactly the case for Morse-Smale diffeomorphisms.

Nevertheless, in many situations these spectra can have a very complicated fractal structure. An important example in the theory appears naturally from number theory, more specifically from the theory of Diophantine approximations and quadratic forms as follows: given a positive real number $\alpha$ we define its \textit{best constant of Diophantine approximation} to be
\begin{align*}
k(\alpha) = \sup \left\{
        k>0\colon\,
        \left|\alpha-\dfrac{p}{q}\right|<\dfrac{1}{kq^2} \text{ has infinitely many solutions } \dfrac{p}{q} \in \Q
    \right\}.%\\
%&=& \limsup_{p \in \Z,q\in \N, |p|,q \to \infty} |q(q\alpha -p)|^{-1} \in \R \cup \{+\infty\}.  
\end{align*}
The \emph{classical Lagrange spectrum}, denoted by $L$, is the collection of the quantities $k(\alpha)$ which are finite. The \emph{classical Markov spectrum} is defined as the set
\begin{align*}
    M = \left\{
        \inf_{(x,y) \in \Z^2 \setminus (0,0)} |f(x,y)|^{-1}\colon\,
        f(x,y)=ax^2+bxy+cy^2  \text{, with } b^2-4ac=1
    \right\}.
\end{align*}
The link between the classical notions and the dynamical setting is provided by the following characterization: let $\sigma: \N^{\Z} \to \N^{\Z}$ be the shift map and $f: \Sigma \to \R$ be a continuous function given by $f(\theta)=[\theta_0; \theta_1, \theta_2, \ldots]+[0; \theta_{-1}, \theta_{-2}, \ldots]$, $\theta = (\theta_n)_{n\in \Z}$. Then we have,
\begin{align*}
    L = L(\sigma,\, f)
    \quad
    \text{and}
    \quad
    M = M(\sigma,\, f).
\end{align*}
It is also possible to recover the classical Markov and Lagrange spectrum from a smooth setting. Indeed, let $\psi:(0,1)^2 \to (0,1)^2$ be defined by 
    $$\psi (x,y) =\left(\left\{ \dfrac{1}{x}\right\}, \dfrac{1}{\lfloor 1/x\rfloor+y} \right),$$
where $\{1/x\}$ is the fractional part of $1/x$ and $\lfloor1/x\rfloor = 1/x - \{1/x\}$. Define the set $C_N=\{[0;a_1, a_2, \ldots]: \; 1 \leq a_n \leq N\}$ and so $\Lambda_N=C(N) \times C(N)$ is a compact $\psi$-invariant subset of $(0,1)^2$. Given $h: (0,1)^2 \to \R$ defined by $h(x,y)=1/x+y$, we have
\begin{align*}
    L\cap (-\infty, N)=L(\Lambda_N,\, h)\cap (-\infty, N)
    \quad
    \text{and}
    \quad
     M\cap (-\infty, N)=M(\Lambda_N,\, h)\cap (-\infty, N).
\end{align*}

Another classical spectrum also coming from Diophantine approximations is called Dirichlet spectrum. An approach given by Davenport and Schmidt \cite{DS} allows us to define this set in terms of the shift map $\sigma: \Sigma \to \Sigma$ as the set $D:= \left\{
        \limsup_{j\to\infty}\, g(\sigma^j(\theta))\colon\,
        \theta \in \Sigma
\right\}$, where $g: \Sigma \to \R$  is given by $g(\theta)=[\theta_0; \theta_1, \theta_2, \ldots]\cdot[\theta_{-1}; \theta_{-2}, \theta_{-3}, \ldots]$. Analogously, we are able to see this spectrum as a dynamically defined Lagrange spectrum in the smooth setting. Choosing $s: (0,1)^2 \to \R$ given by $s(x,y)=1/xy$, we have that
\begin{align*}
    D\cap (-\infty, N)=L(\Lambda_N,\, s)\cap (-\infty, N),
\end{align*}
where $\Lambda_N$ is the above horseshoe for the map~$\psi$.

%%%%%%%%%%%%%%%%%%%%%%%%%%%%%%%%%%%%%%%%%%%%%%%%%%%%%%
\subsection{Horseshoes}

Unless otherwise stated in this article $M$ denotes a smooth $d$-dimensional compact riemannian manifold with $d\geq 2$. We write $\diff^r(M)$ to denote the space of diffeomorphisms from $M$ to $M$ of class $C^r$, where $r\in [1,\infty]$.

\subsubsection{Horseshoes and holonomies}
Consider $\varphi\in \diff^r(M)$, $r\in [1,\infty]$. We say that a $\varphi$-invariant compact set $\Lambda \subset M$ is \emph{hyperbolic} if there exists a decomposition of the tangent bundle of $\Lambda$ on two continuous, $D\varphi$-invariant, sub-bundles $E^s$ and $E^u$, i.e., $T\Lambda = E^s\oplus E^u$, and there exist constants $C>0$ and $\lambda \in (0,1)$ such that
\begin{align*}
    \norm{D\varphi^n(x)|_{E^s(x)}}\leq C\lambda^n
    \quand
    \norm{D\varphi^{-n}(x)|_{E^u(x)}}\leq C\lambda^n,
\end{align*}
for every $n\geq 0$ and every $x\in \Lambda$. A hyperbolic set $\Lambda$ is a \emph{horseshoe} for $\varphi$ if $\Lambda$ is infinite, totally disconnected, transitive and $\overline{\per(\varphi|_{\Lambda})} = \Lambda$.

An important feature ensured by the hyperbolic structure in the set $\Lambda$ is the existence of \emph{stable and unstable laminations}. More precisely, for each point $x\in \Lambda$ there exists a pair of traversals $\varphi$-invariant, immersed $C^r$-submanifolds $W^s(x)$ and $W^u(x)$ which are tangent respectively to $E^s(x)$ and $E^u(x)$ at $x$. These are called \emph{stable and unstable manifolds} of $\varphi$ at $x$. The \emph{stable lamination} of $\Lambda$, $W^s(\Lambda)$ is defined as the collection of $W^s(x)$, with $x\in \Lambda$. Analogously we define the \emph{unstable lamination} of $\Lambda$, $W^u(\Lambda)$. 

For $C^1$-diffeormorphisms the Stable Manifold Theorem \cite{Sh2013} ensures that the stable and unstable manifolds are locally embedded disks of dimension $\dim E^s$ and $\dim E^u$, respectively. Moreover, for any positive small $\delta$ and any point $x\in \Lambda$ if we denote by $W^s_{\delta}(x)$ ($W^u_{\delta}(x)$) the set of points $y\in W^s(x)$ ($y\in W^u(x)$) with $d(x,y)<\delta$ ($d$ is the distance associated to a riemannian structure on $M$), then the correspondences that associates each $x\in \Lambda$ to the local stable manifold $W^s_{\delta}(x)$ and the local unstable manifold $W^u_{\delta}(x)$ are continuous. In other words, we have continuity of the stable and unstable laminations.
Assuming that the diffeomorphism is $C^r$ with $r>1$, Pugh, Shub and Wilkinson in \cite{PuShWi1997} proved the stable and unstable laminations are H\"older continuous, but we cannot expect to be much more than that in general. There are examples where these maps are not even Lipschitz (see~\cite[Section 3]{PaVi1994}). 

A good regularity of the laminations shows to be particularly useful in the analysis of the the fractal properties of the horseshoe. Using the fact that the horseshoe $\Lambda$ is locally (homeomorphic to) products of the form $W^s_{\delta}(x)\cap \Lambda\times W^u_{\delta}(x)\cap \Lambda$, the good regularity of the laminations ensures a certain type of ``fractal homogeneity" which allows us to focus in the structure of the \emph{stable and stable Cantor sets} $W^s_{\delta}(x)\cap \Lambda$, $W^u_{\delta}(x)\cap \Lambda$, in any point $x\in \Lambda$, to obtain fractal properties of $\Lambda$. This is exactly the case when the ambient manifold $M$ is two dimensional where it is known, \cite{PaTa1995}, that is possible to extend the stable and unstable laminations to $C^{1+\alpha}$-foliations in a neighborhood of $\Lambda$. This extension is the initial technical step for the study of two-dimensional systems presenting highly relevant dynamical phenomena (see for example, \cite{MoYo2010}, \cite{PaYo1994}, \cite{PaTa1995} and references therein).

The regularity of the stable and unstable lamination is intrinsically related with regularity of the \emph{holonomy maps} defined by these laminations (See the discussion in \cite{PuShWi1997}). \emph{Unstable holonomies} are defined as follows: for $x_0, y_0\in \Lambda$ with $y_0\in W^u_{\delta}(x_0)$ set $H^u_{x_0,y_0}: W^s_{\delta}(x_0)\cap \Lambda\to W^s_{\delta}(y_0)\cap\Lambda$ as the map that sets each $z\in W^s_{\delta}(x_0)\cap \Lambda$ into the unique intersection point $H^u_{x_0,y_0}(z)$ between $W^u_{\delta}(z)\cap W^s_{\delta}(y_0)$. The local product structure of the horseshoe $\Lambda$ guarantees that $H^u_{x_0,y_0}(z)\in \Lambda$ and so $H^u_{x_0,y_0}$ is well defined. Analogously, we define \emph{stable holonomies}.

\subsubsection{Regular Cantor sets}
In this work we deal with subsets of the stable Cantor sets $W^s_{\delta}(x)\cap \Lambda$ which have a regular structure. The model of such concept is provided by the notion of regular Cantor set on the real line. A compact set $K\subset \R$ is a \emph{regular Cantor set} if there exist $\gamma>0$, a cover of $K$ by disjoint intervals $I_1,\dots,I_k$ and a $C^{1+\gamma}$ expanding function $\tau:\cup_{i=1}^kI_i\rightarrow \cup_{i=1}^kI_i$ satisfying that
\begin{enumerate}
    \item For every $1\leq i\leq k$, there exists $1\leq j\leq k$ such that $\tau(I_j)\supset I_i$;
    
    \item $|\tau'(t)|>1$, for every $t\in I_j$ and for every $j=1,\dots, k$;
    
    \item For every $1\leq j\leq k$ and every $n$ sufficient large, $\tau^n(I_j) \supset \bigcup_{i=1}^kI_i$;
    
    \item $K = \bigcap_{n\geq 0}\tau^{-n}(\bigcup_{i=1}^{k} I_i)$.
\end{enumerate}
In \cite{PaTa1995}, we can see that if $K$ is a regular Cantor sets, then $0< \dim_H(K)<1$ (the same is not true if we extend the notion of regular Cantor sets allowing that the map $\tau$ can be taken $C^1$). Another important feature of regular Cantor sets is that its Hausdorff dimension varies continuously with respect to the expanding map $\tau$.

\subsubsection{Dominated splitting}
Let $\Lambda$ be a compact $\varphi$-invariant set and $F$ be a continuous $D\varphi$-invariant linear fiber bundle over $\Lambda$. We say that $F$ has a \emph{dominated splitting} if there exists a decomposition $F = F_1\oplus F_2\oplus \cdots \oplus F_k$ into continuous $D\varphi$-invariant sub-bundles $F_1,\dots, F_k$ over $\Lambda$ and there exist constants $C>0$ and $\lambda\in (0,1)$ such that
\begin{align*}
    \norm{D\varphi^n(x)|_{F_{i+1}(x)}}
    \leq C\lambda^n\norm{(D\varphi^{n}(x)|_{F_i(x)})^{-1}}^{-1},
\end{align*}
Notice that for hyperbolic sets the tangent bundle $T\Lambda$ has a dominated splitting given by $T\Lambda = E^s\oplus E^s$. If additionally $E^s$ has a dominated splitting of the form $E^s = E^w\oplus E^{ss}$ we refer to the bundles $E^w$ and $E^{ss}$ as \emph{weak-stable bundle} and \emph{strong-stable bundle}. Similar terminology is used when $E^u$ has a dominated splitting.

%%%%%%%%%%%%%%%%%%%%%%%%%%%%%%%%%%%%%%%%%%%%%%%%%%%%
\subsection{Symbolic dynamics}
\label{subsection:symbolic dynamics}
In this subsection we describe subshifts of finite type. This is the class of symbolic systems that provides a combinatorial way to interpret the dynamics of hyperbolic systems.

\subsubsection{Subshifts of finite type}
Let $\mathbb{A}$ be a finite set that we call \emph{alphabet} and let $B = (B_{\alpha\beta})_{\alpha,\beta\in \mathbb{A}}$ be a transition matrix , i.e., the entries of $B$ satisfies $B_{\alpha\beta} \in \{0,1\}$ for every $\alpha,\beta\in \mathbb{A}$. We say that the pair $(\alpha,\beta)$ is admissible (or $B$-admissible) if $B_{\alpha\beta} = 1$.

A finite word $ (\theta_1,\ldots, \theta_n)\in \mathbb{A}^n$ (or a sequence $(\theta_i)_{i\in\Z} \in \mathbb{A}^{\Z}$) is admissible, if $(\theta_i,\theta_{i+1})\in \mathbb{A}^2$ is admissible for every $i$. Denote by $\Sigma$ the subset of $\mathbb{A}^{\Z}$ formed by the admissible sequences. We refer to the pair $(\sigma,\, \Sigma)$ as the \emph{subshift of finite type} associated to $B$, where $\sigma:\mathbb{A}^{\Z}\to\mathbb{A}^{\Z}$ denotes the shift map.

Throughout this article we assume that $B$ is a transitive matrix, meaning that for every pair $(\alpha,\beta)\in \mathbb{A}^2$, there exists $j\in \N$ such that $B_{\alpha\beta}^j >0$. This implies the transitivity of the shift  $\sigma:\Sigma\to\Sigma$. In particular, for each pair of symbols $(\alpha, \beta)\in \mathbb{A}^2$, there exists a finite admissible word $c$ such that the word $(\alpha, c, \beta)$ is admissible (sometimes we also write $\Sigma$-admissible). We refer to such word $c$ is a \emph{gluing word} connecting $\alpha$ and $\beta$.

The notion of topologically mixing for a subshift finite type $(\sigma,\, \Sigma)$ can also be characterized in terms of the transition matrix $B$, namely it is equivalent to $B$ being \emph{aperiodic} meaning that there is a $m$ such that such that all entries of the matrix $B^m$ are positive. 

Given a sequence $\theta = (\theta_n)_{n\in\Z}\in \mathbb{A}^{\Z}$, we write
\begin{align*}
    \theta^-:= (\ldots, \theta_{-n},\ldots, \theta_{-1},\theta_0)
    \quad
    \text{and}
    \quad
    \theta^+ := (\theta_1,\ldots, \theta_n,\ldots).
\end{align*}
We also may write $\theta = (\theta^{-\ast}, \theta^+)$, where $\ast$ indicates the zero-th position of the sequence $\theta$ which in this case is placed at the symbol $\theta_0$.

Given a finite word $a = (a_1,\ldots, a_n)\in \mathbb{A}^n$ and $m\in \Z$ we denote by
\begin{align*}
    C_m(a) = \left\{
        \theta\in \mathbb{A}^{\Z}\colon
        \theta_{m-1+i} = a_i,\, \text{for all }\, 1\leq i\leq n
    \right\},
\end{align*}
the \emph{cylinder} defined by $a$ starting at $m$. Similarly, we can define the cylinder associated with a infinite sequence $a\in \mathbb{A}^{\N}$. In the case that we are working with a subshift of finite type $(\Sigma, \sigma)$ and $a\in \mathbb{A}^n$ is a finite $\Sigma$-admissible word, $C_m(a)$ represents a subset of $\Sigma$. Another simplification that we adhere is dropping the word admissible once the subshift that we are working is fixed.

For each sequence $\theta\in \Sigma$ we set $W^s_{\text{loc}}(\theta) := \{\zeta\in \Sigma\colon\, \zeta^+ = \theta^+,\, \zeta_0=\theta_0\}$ and $W^u_{\text{loc}}(\theta) := \{\zeta\in \Sigma\colon\, \zeta^- = \theta^-\}$.  For each pair of sequences $\theta,\zeta\in \Sigma$ with $\theta_0 = \zeta_0$, the intersection $W^s_{\text{loc}}(\zeta)\cap W^u_{\text{loc}}(\theta)$ consists of a single point denoted by the brackets $[\zeta, \theta] := (\theta^{-*},\zeta^+)$. The \emph{symbolic unstable holonomy} between two sequences $\theta$ and $\zeta$, with $\zeta\in W^u_{{loc}}(\theta)$, is defined as the map $h^u_{\theta,\zeta}:W^s_{\text{loc}}(\theta)\to W^s_{\text{loc}}(\zeta)$, $h^u_{\theta,\zeta}(\xi) = [\zeta, \xi]$. Symbolic stable holonomy is defined similarly.

If $\xi\in W^s_{\text{loc}}(\theta)$, then $\sigma^{-n}(\xi)\in W^s_{\text{loc}}(\sigma^{-n}(\theta))$ if and only if $\xi_{-i} = \theta_{-i}$, for every $0\leq i \leq n-1$. Under this conditions, we have the following invariance property of the holonomies: restricted to $W^s_{\text{loc}}(\theta)\cap C_{-n+1}((\theta_{-n+1},\ldots, \theta_0))$, 
\begin{align}\label{eq:simbolicHolonomyInvariance}
    \sigma^n\circ h^u_{\sigma^{-n}(\theta),\sigma^{-n}(\zeta)}\circ \sigma^{-n} = h^u_{\theta,\zeta}.
\end{align}

\subsubsection{Markov partition for horseshoes}
\label{subsection:markovPartition}

Let $\Lambda$ be a horseshoe for a diffeomorphism $\varphi\in \diff^r(M)$, $r\geq 1$, and let $\delta>0$ be a small real number. A \emph{rectangle} $P$ is a subset of $\Lambda$ with $\diam(P)\leq \delta$, such that for every $x,y\in P$,
\begin{align}\label{eq:090323.1}
    \{[x,y]_{\delta}\} := W^s_{\delta}(x)\cap W^u_{\delta}(y) \subset P.
\end{align}
A rectangle $P\subset \Lambda$ is said to be \emph{proper} if $\overline{\interior_{\Lambda}(P)} = P$. A \emph{Markov partition} \linebreak $\mathcal{P} = \{P(\alpha);\, \alpha\in \mathbb{A}\}$ for $\Lambda$ is a finite partition of $\Lambda$ by proper rectangles satisfying that for every $\alpha,\beta\in \mathbb{A}$, $\interior_{\Lambda}(P(\alpha))\cap \interior_{\Lambda}(P(\beta)) = \varnothing$ and if  $x\in \interior_{\Lambda}(P(\alpha))\cap \varphi^{-1}(\interior_{\Lambda}(P(\beta)))$, then
\begin{align*}
    \varphi(W^s_{\delta}(x)\cap P(\alpha))\subset W^s_{\delta}(\varphi(x))\cap P(\beta)
    \quad
    \text{and}
    \quad
    \varphi(W^u_{\delta}(x)\cap P(\alpha))\supset W^u_{\delta}(\varphi(x))\cap P(\beta).
\end{align*}
The finite set $\mathbb{A}$ is called \emph{alphabet} associated with the Markov partition $\mathcal{P}$. By Bowen, \cite{Bo2008}, given a horseshoe we can find Markov partitions with arbitrarily small diameters $\delta>0$.

Let $\mathcal{P} = \{P(\alpha);\ \alpha\in \mathbb{A}\}$ be a Markov partition for the horseshoe $\Lambda$. We say that a pair $(\alpha,\beta)\in \mathbb{A}\times\mathbb{A}$ is \emph{admissible} (or $\Lambda$-admissible) if $\varphi(P(\alpha))\cap P(\beta)\neq \varnothing$. Set,
\begin{align*}
    \Sigma = \left\{
        (\theta_i)_{i\in \Z} \in \mathbb{A}^{\Z};\
        (\theta_i, \theta_{i+1}) \text{ is admissible for every } i\in \Z 
    \right\},
\end{align*}
and note that $\Sigma$ is invariant by the shift map $\sigma:\mathbb{A}^{\Z}\to\mathbb{A}^{\Z}$ and by the transitivity of the horseshoe, $\sigma|_{\Sigma}$ is transitive. Moreover, there exists a bi-H\"older map ( \cite[Theorem 19.1.2]{KaHa1995}) $\Phi_{\varphi}:\Sigma\rightarrow \Lambda$ which conjugates $\varphi|_{\Lambda}$ and $\sigma|_{\Sigma}$, i.e.,  $\Phi_{\varphi}\circ \sigma|_{\Sigma} = \varphi\circ \Phi_{\varphi}$.  In this case, the transition matrix $B = (B_{\alpha\beta})$ is determined by the condition that
\begin{align*}
    B_{\alpha\beta} = \left\{
        \begin{array}{cl}
            1, & \varphi(P(\alpha))\cap P(\beta) \neq \varnothing \\
            0, & \text{ otherwise}.
        \end{array}
    \right.
\end{align*}

We can use the map $\Phi_{\varphi}$ to relate the dynamics of the horseshoe with its symbolic counterpart. For instance, $\Phi_{\varphi}$ relates the previously defined notions of brackets. Indeed, first observe that there exists $n_{\delta}$ such that for any finite word $a \in \mathbb{A}^{2n_{\delta}+1}$ and any $\theta \in C_{-n_{\delta}}(a)$, we have that the restrictions $\Phi_{\varphi}: W^s_{\text{loc}}(\theta)\cap C_{-n_{\delta}}(a)\to W^s_{\delta}(\Phi_{\varphi}(\theta))\cap \Lambda$ and $\Phi_{\varphi}: W^u_{\text{loc}}(\theta)\cap C_{-n_{\delta}}(a)\to W^u_{\delta}(\Phi_{\varphi}(\theta))\cap \Lambda$ are well defined. So, for $\theta,\zeta\in C_{-n_{\delta}}(a)$, 
\begin{align}\label{eq:090323.2}
    \left[
        \Phi_{\varphi}(\theta),\,
        \Phi_{\varphi}(\zeta)
    \right]_{\delta} = \Phi_{\varphi}\left(\left[
        \theta,\,
        \zeta
    \right]\right).
\end{align}

Notice that if $x_0,y_0$ are in the same rectangle of the Markov partition $\mathcal{P}$, say $\mathcal{P}(\alpha)$, with $y_0\in W^u_{\delta}(x_0)$ we may use the brackets notation to have an explicit expression to the unstable holonomy $H^u_{x_0,y_0}:W^s_{\delta}(x_0)\cap \mathcal{P}(\alpha)\to W^s_{\delta}(y_0)\cap\mathcal{P}(\alpha)$ by $H^u_{x_0,y_0}(z) = [y_0,z]_{\delta}$. Analogously we can express the stable holonomy using the brackets notation.

Take $x_0,y_0\in\Lambda$ with $y_0\in W^u_{\delta}(x_0)$, and consider $H^u_{x_0,y_0}$ the unstable holonomy between $W^s_{\delta}(x_0)\cap \Lambda$ and $W^s_{\delta}(y_0)\cap \Lambda$. For any finite word $a \in \mathbb{A}^{2n_{\delta}+1}$, if $\xi\in W^s_{\text{loc}}(\Phi^{-1}_{\varphi}(x_0))\cap C_{-n_{\delta}}(a)$, then $\Phi_{\varphi}(\xi) \in W^s_{\delta}(x_0)\cap\Lambda$ and so, by equation \eqref{eq:090323.2},
\begin{align}\label{eq:090323.3}
    H^u_{x_0,y_0}\circ \Phi_{\varphi}(\xi)
    = [y_0,\, \Phi_{\varphi}(\xi)]_{\delta}
    = \Phi_{\varphi}\left(\left[
        \Phi_{\varphi}^{-1}(y_0),\, \xi
    \right]\right) = \Phi_{\varphi}\circ h^u_{\Phi_{\varphi}^{-1}(x_0), \Phi^{-1}_{\varphi}(y_0)}(\xi).
\end{align}
In particular, we can express the invariance of the holonomy in the following way: there exists a decreasing sequence of positive parameters $(\delta_n)_n$, $\delta_0 = \delta$ such that the restriction of $H^u_{x_0,y_0}$ to $W^s_{\delta_n}(x_0)\cap \Lambda$ satisfies
\begin{align}\label{eq:holonomyInvariance}
    H^u_{x_0,y_0} = \varphi^n\circ H^u_{\varphi^{-n}(x_0),\varphi^{-n}(y_0)}\circ \varphi^{-n}.
\end{align}
This is a direct consequence of the symbolic holonomy invariance \eqref{eq:simbolicHolonomyInvariance} and equation \eqref{eq:090323.3} above.

\subsubsection{Symbolic towers}
\label{subsubsection:symbolicTowers}
Let $\Sigma\subset \mathbb{A}^{\Z}$ such that $(\sigma,\, \Sigma)$ is a subshift of finite type associated to a transition matrix $B$. For every finite set of finite $B$-admissible words $F\subset \cup_{j\geq 1}\mathbb{A}^j$ we may associated the subshift of the sequences that are produced by concatenations of words in $F$ in the following way: consider the alphabet $\mathbb{A}(F):= F$ and the transition matrix $B(F) = (B(F)_{a,b})_{a,b\in F}$ is given by
\begin{align*}
    B(F)_{a,b} = \left\{
        \begin{array}{cc}
            1, & (a,b) \text{ is } B-\text{admissible} \\
            0, & \text{ otherwise}.
        \end{array}
    \right.
\end{align*}
Denote by $\sigma(F):\mathbb{A}(F)^{\Z}\to\mathbb{A}(F)^{\Z}$ the shift map and by $\Sigma(F)$  the set of $B(F)$ admissible words in the alphabet $\mathbb{A}$. Notice that we can see $\Sigma(F)$ as a subset $\Sigma$ and with this identification, for every $a\in \mathbb{A}(F)$, we can write $\sigma(F)|_{C_0(a)} = \sigma^{|a|}$, where $|a|$ denotes the size of the word $a$. In particular, if $F\subset \mathbb{A}^k$ consists of $B$-admissible subwords of the same size in the alphabet $\mathbb{A}$, then we simple write $\sigma(F) = \sigma^k$ and $\Sigma(F) \subset \Sigma$.

We say that $\Sigma$ is a \emph{tower} of size $k\in \N$ over a $\sigma^k$-invariant subset $\Sigma_0\subset \Sigma$ if $\Sigma_0 = \Sigma(F)$, where $F = \{ a\in \mathbb{A}^k\colon\, a \text{ is } B\text{-admissible and } C_0(a)\cap \Sigma_0\neq \varnothing\}$.
\begin{example}
\label{ex:symbolicTowers}
\normalfont
    A natural way to produce a tower over a given subshift can be described as follows. Let $\mathbb{A}_0$ be a finite alphabet and $(\sigma_0,\, \Sigma_0)$ be a subshift of finite type on the alphabet $\mathbb{A}_0$ associated to a transition matrix $B_0$. Define the alphabet,
\begin{align*}
    \mathbb{A} := \left\{
        (j,\, \alpha)\colon\,
        0\leq j \leq k-1
        \quad
        \text{and}
        \quad
        \alpha \in \mathbb{A}_0
    \right\},
\end{align*}
and transition matrix $B$ on the alphabet $\mathbb{A}$ by the rule that a pair of symbols in $\mathbb{A}$, $((j,\alpha),(i,\beta))$ is $B$-admissible, if $0\leq j < k-1$, $i = j + 1$ and $\beta = \alpha$ or $j = k-1$, $i = 0$ and $(\alpha,\, \beta)$ is $B_0$-admissible. Let $(\sigma,\, \Sigma)$ be the subshift of finite type on the alphabet $\mathbb{A}$ associate to $B$. Then, it is natural to see $\Sigma$ as a tower of size $k$ for $\Sigma_0$. Indeed, $\Sigma_0$ can be identified with $\Sigma(F)\subset \mathbb{A}^k$ where $F = \{((0,\alpha),\ldots, (k-1,\alpha))\colon\, \alpha \in \mathbb{A}_0\}$ and $\sigma_0$ can be identified with $\sigma^k$.
\end{example}
\begin{proposition}
\label{prop:symbolicStructureForTowerHorseshoe}
    Let $\varphi\in \diff^r(M)$ and let $\Lambda_0\subset M$ be a horseshoe associated to $\varphi^k$, $k\geq 2$. Then, the set $\Lambda = \cup_{j=0}^{k-1}\varphi^{-j}(\Lambda_0)$ is a horseshoe for $\varphi$. Moreover, there exist an alphabet $\mathbb{A}$ and sets $\Sigma,\, \Sigma_0\subset \mathbb{A}^{\Z}$ such that $\Sigma = \cup_{j=0}^{k-1}\sigma^{-j}(\Sigma_0)$ and $(\sigma,\, \Sigma)$, $(\sigma^k,\, \Sigma_0)$ are subshifts of finite type conjugated respectively to $(\varphi,\, \Lambda)$ and $(\varphi^k,\, \Lambda_0)$ and $\Sigma$ is a tower of size $k$ over $\Sigma_0$.
\end{proposition}
\begin{proof}
    Consider an alphabet $\mathbb{A}_0$ and a subshift of finite type $(\sigma_0, \Sigma_0)$ on the alphabet $\mathbb{A}_0$ conjugated to the horseshoe $(\varphi^k,\, \Lambda_0)$. Using the construction given in the Example \ref{ex:symbolicTowers} we can produce a subshift of finite type $(\sigma,\, \Sigma)$ such that $\Sigma_0$ can be seen as a $\sigma^k$-invariant subset of $\Sigma$ and $\sigma_0 = \sigma^k$. It is direct to check that, $(\sigma,\, \Sigma)$ is conjugated to $(\varphi,\, \Lambda)$ and $\Sigma$ is a tower of size $k$ over $\Sigma_0$.
\end{proof}

%%%%%%%%%%%%%%%%%%%%%%%%%%%%%%%%%%%%%%%%%%%%%%%%%%%%
\subsection{Intrinsic derivatives}
In many moments in this article we deal with functions defined only on compact sets (such as the ones defined on stable/unstable Cantor sets) aiming to extract regularity properties that allow us to measure geometric distortions of these sets in small scales.

The differentiability of maps defined in compact sets was comprehensively discussed in \cite{Wh1992}, in which extension theorems were obtained guaranteeing that those maps are differentiable in the classical sense. However, the content of this subsection is a collection of results established in \cite{PaVi1994}, added here for the convenience of the reader.

Let $X$ be a compact subset of $\R^m$. We say that a function $F:X\to\R^n$ is \emph{intrinsically differentiable} (or shorter $IC^1$) if there exists a continuous map $\Delta F: X\times X\to \M_{n\times m}(\R)$ such that for every $x,y\in X$:
\begin{align*}
    F(x) - F(y) = \Delta F(x,y)\, (x-y).
\end{align*}
If additionally the map $\Delta F$ is $\gamma$-H\"older continuous, for some $\gamma\in (0,1]$, we say that $F$ is $IC^{1+\gamma}$. It is important to point out that differently of classical derivatives, the intrinsic derivative, when exists, does not need to be unique.

Classical derivatives and intrinsic derivatives share many basic properties. See~\cite[Section 2]{PaVi1994} for precise statements and proofs of these properties that we make use here without further reference. These similarities are natural due to the fact that intrinsic differentiable functions on compact sets are restrictions of classical differentiable functions defined in slightly bigger open sets. See \cite{Wh1992}.

For a compact set $X\subset \R^m$ we also define the \emph{intrinsic tangent space} of $X$ at a point $x\in X$, $IT_xX$ as the span of the directions $v\in \R^m$ such there exists a compact set $K(v)\subset \R$, $0\in K(v)$, and a $IC^1$ map $\gamma:K(v)\to X$ such that $\Delta\gamma(0,0)\cdot 1 = v$. Notice that if $F:X\to \R^n$ is $IC^1$, then for every $z\in X$, the intrinsic derivative $\Delta F(z,z)$ sends $IT_zX$ into $IT_{F(z)}F(X)$.

%%%%%%%%%%%%%%%%%%%%%%%%%%%%%%%%%%%%%%%%%%%%%%%%%%%%%
\subsection{Linearization}
Given $T\in \GL_d(\R)$ a hyperbolic matrix of saddle-type, i.e., $\spec(T) := \spec(T)\cap \{z\in \C\colon |z|<1\} \neq \varnothing$ and $\spec(T) := \spec(T)\cap \{z\in \C\colon |z|>1\} \neq \varnothing$, we define the \emph{spectral spreading} of $T$ by the quantities
\begin{align*}
    \rho^{\pm}(T) := \frac{
        \max\left\{
            |\log|\lambda||\colon\,
            \lambda \in \spec^{\pm}(T)
        \right\}
    }{
        \min\left\{
            |\log|\lambda||\colon\,
            \lambda \in \spec^{\pm}(T)
        \right\}
    }.
\end{align*}
Given a natural number $r\geq 1$, we define the $r$-\emph{smoothness} of $T$ by
\begin{align*}
    K^r(T) := \left\lfloor
        \max_{m,n\in \N\colon\, m+n = r} \min\left\{
            \frac{m}{\rho^+(T)},
            \frac{n}{\rho^-(T)}
        \right\}
    \right\rfloor.
\end{align*}
Let $\lambda_1,\cdots, \lambda_d$ be the eigenvalues of $T$ counted with multiplicity and consider the function $\mu:\C\times\Z^d\to \C$ given by $\mu(\lambda, l) = \lambda\cdot(\lambda_1^{l_1}\cdots \lambda_d^{l_d})^{-1}$. We say that $T$ satisfies the \emph{resonant condition of order} $r$ if $|\mu(\lambda, l)| \neq 1$, for every $l\in \Z^d$ with $2\leq |l| \leq r$ and for every $\lambda \in \spec(T)$.
\begin{proposition}
\label{prop:resonanceParametersProperties}
    It holds that,
    \begin{enumerate}
        \item\label{prop:resonanceParametersProperties-item1} The maps $T\mapsto \rho^+(T)$ and $T\mapsto \rho^-(T)$ are continuous;
        \item\label{prop:resonanceParametersProperties-item3} The set of matrices satisfying the resonant condition of order $r$ is open and dense on $\GL_d(\R)$.
    \end{enumerate}
\end{proposition}
\begin{proof}
    The proof of the proposition follows directly from the continuity of the eigenvalues with respect to the matrices (for item \ref{prop:resonanceParametersProperties-item3} notice that the resonant condition of order $r$ only imposes a finite number of constraints in the set of eigenvalues).
\end{proof}

Let $\varphi:M\to M$ be a $C^r$-diffeomorphism, $r\geq 1$ and consider $p\in M$ a hyperbolic periodic point of saddle-type.  We write $\rho^{\pm}_p(\varphi)$ and $K_{p}^r(\varphi)$ to denote respectively $\rho^{\pm}(D\varphi^{\per(p)}(p))$ and $K^r(D\varphi^{\per(p)}(p))$. We also say that $(\varphi, p)$ satisfies the resonant condition of order $r$ if $D\varphi^{\per(p)}(p)$ satisfies the resonant condition of order $r$.

The diffeomorphism $\varphi$ admits a $k$-\emph{linearization} around $p$, if there exists a $C^k$-chart $h:U\to \R^d$ centered at $p$, such that $h \circ \varphi^{\per(p)}(x) = T \circ h(x)$, for every $x \in U \cap \varphi^{-\per(p)}(U)$, where $T$ is the expression of $D\varphi^{\per(p)}(p)$ in this system of coordinates. The next result give conditions for $C^k$-linearization in terms of $K_p^r(\varphi)$.
\begin{theorem}[Sell, \cite{Se1985}]
\label{thm:sellTheorem}
    Let $\varphi\in \diff^{3r}(M)$, $r\geq 2$, admitting a hyperbolic periodic point of saddle-type $p$. Assume that $(\varphi, p)$ satisfies the resonant condition of order $r$. Then, $\varphi$ admits a $K^r_p(\varphi)$-linearization around $p$.
\end{theorem}

For any $1\leq r \leq \infty$, let $\mathcal{H}^r(M)$ be the set of diffeomorphisms $\varphi\in \diff^r(M)$ admitting a hyperbolic periodic point of saddle-type. Consider $\varphi \in \mathcal{H}^r(M)$ and let $p\in M$ be such hyperbolic point. Define $r_p(\varphi) := \lfloor 3(\rho^+_p(\varphi)+\rho^-_p(\varphi) + 1)\rfloor$. We also define the subset $\mathcal{L}^r(M)$ of $\mathcal{H}^r(M)$ of diffeomorphisms $\varphi$ admitting a hyperbolic periodic point of saddle-type $p\in M$ such that
\begin{enumerate}
    \item\label{condition1:resonantCondition} $(\varphi, p)$ satisfies the resonant condition of order $r_p(\varphi)$;

    \item\label{condition2:regularityBound} $r_p(\varphi) < r/3$.
\end{enumerate}

\begin{proposition}
\label{prop:DiffResonantCondition}
    Assume $r\geq 6$.
    \begin{enumerate}
        \item\label{prop:DiffResonantCondition-item1} If $\varphi\in \mathcal{L}^r(M)$ and $p\in M$ is a periodic point of $\varphi$ satisfying \ref{condition1:resonantCondition} and \ref{condition2:regularityBound}, then $\varphi$ admits a $3$-linearization around $p$;

        \item\label{prop:DiffResonantCondition-item2} The set $\mathcal{L}^r(M)$ is $C^1$-open in $\diff^r(M)$ and $C^r$-dense on the set of $\varphi\in \mathcal{H}^r(M)$ admitting a hyperbolic periodic point of saddle-type $p\in M$ with $r_p(\varphi)< r/3$. In particular, $\mathcal{L}^{\infty}(M)$ is $C^1$-open and $C^{\infty}$-dense in $\mathcal{H}^{\infty}(M)$.
    \end{enumerate}
\end{proposition}
\begin{proof}
    To see item \ref{prop:DiffResonantCondition-item1} notice that we can apply Theorem \ref{thm:sellTheorem} with the regularity $r':=r_p(\varphi)$. Indeed, since $\varphi\in \diff^r(M)$, $r\geq 6$, condition \ref{condition2:regularityBound} guarantees that $\varphi\in \diff^{3r'}(M)$ with $r'\geq 2$. Furthermore, condition \ref{condition1:resonantCondition} guarantees that $(\varphi, p)$ satisfies the resonant condition of order $r'$. Therefore, $\varphi$ admits a $K^{r'}_p(\varphi)$-linearization around $p$. But, by the definition of $r_p(\varphi)$ we see that $K^{r'}_p(\varphi) \geq 3$.

    To see that $\mathcal{L}^r(M)$ is open in $\mathcal{H}^r(M)$, take $\varphi\in \mathcal{L}^r(M)$ with $p$ satisfying conditions \ref{condition1:resonantCondition} and \ref{condition2:regularityBound}. By item \ref{prop:resonanceParametersProperties-item3} of Proposition \ref{prop:resonanceParametersProperties}, we can consider an open neighborhood $\mathcal{U}(\varphi)\subset \diff^r(M)$, of $\varphi$, such that any $\tilde{\varphi}\in \mathcal{U}(\varphi)$ has a periodic point $\tilde{p}$ which is the hyperbolic continuation of $p$ and $(\tilde{\varphi},\tilde{p})$ satisfies the resonant condition of order $r_p(\varphi)$. Additionally, we claim that up to decrease the neighborhood $\mathcal{U}(\varphi)$ we have that for every $\tilde{\varphi}\in \mathcal{U}_p(\varphi)$,
    \begin{align}
    \label{eq:011023.1}
        r_p(\varphi) - 1 \leq r_p(\tilde{\varphi}) \leq r_p(\varphi).
    \end{align}
    To see that, consider the continuous map $\mathcal{U}(\varphi)\ni \tilde{\varphi}\mapsto \gamma(\tilde{\varphi}) := 3(\rho^+(\tilde{\varphi}) + \rho^-(\tilde{\varphi}) + 1)$. Up to decrease $\mathcal{U}(\varphi)$, we can assume that for every $\tilde{\varphi}\in \mathcal{U}(\varphi)$,
    \begin{align*}
        |\gamma(\tilde{\varphi}) - \gamma(\varphi)|\leq \left\{
        \begin{array}{cc}
           \frac{1}{2}\min_{m\in\Z}|\gamma(\varphi) - m|  & \text{ if } \gamma(\varphi)\notin \Z\\
            1/2 & \gamma(\varphi)\in \Z.
        \end{array}
        \right.
    \end{align*}
    Thus, for any $\tilde{\varphi}\in \mathcal{U}(\varphi)$,
    \begin{align*}
        r_p(\varphi) - 1
        = \lfloor\gamma(\varphi)\rfloor - 1
        \leq \lfloor\gamma(\tilde{\varphi})\rfloor
        = r_p(\tilde{\varphi})
        \leq \lfloor\gamma(\varphi)\rfloor
        = r_p(\varphi).
    \end{align*}
    Thus, for every $\tilde{\varphi}\in \mathcal{U}(\varphi)$ we have $(\tilde{\varphi},\tilde{p})$ satisfies the resonant condition of order $r_{\tilde{p}}(\tilde{\varphi})$ and $r_{\tilde{p}}(\tilde{\varphi})\leq r_p(\varphi) \leq r/3$ guaranteeing that $\mathcal{U}(\varphi)\subset \mathcal{L}^r(M)$.

    For the density, consider $\varphi\in \mathcal{H}^r(M)$ with a hyperbolic periodic point of saddle-type satisfying $r_p(\varphi) \leq r/3$. Using again item \ref{prop:resonanceParametersProperties-item3} of Proposition \ref{prop:resonanceParametersProperties}, we can find $\tilde{\varphi}$ $C^r$-close to $\varphi$ such that if $\tilde{p}\in M$ is the hyperbolic continuation of $p$ for $\tilde{\varphi}$, then $(\tilde{\varphi}, \tilde{p})$ satisfies the resonant condition of order $r_p(\varphi)$. Moreover, we can assume that $r_{\tilde{p}}(\tilde{\varphi})$ satisfies the inequality \eqref{eq:011023.1} and the condition \ref{condition1:resonantCondition} is fulfilled. Also, it follows directly by inequality \eqref{eq:011023.1} that $r_p(\varphi)<r/3$ is a $C^1$-open condition and so for $\tilde{\varphi}$ close enough to $\varphi$ we have that condition \ref{condition2:regularityBound} is immediate. Therefore, $\tilde{\varphi}\in \mathcal{L}^r(M)$. This finishes the proof of item \ref{prop:DiffResonantCondition-item2}. 
\end{proof}

\section{Lagrange for symbolic dynamics}
\label{section:lagrangeSymbolic}

In this section, we describe the ideas introduced in Section 4 of \cite{IBMO2017} using a purely symbolic/combinatorial approach.

Let $(\sigma,\, \Sigma)$ be a subshift of finite type associated to a transition matrix $B$ and $g:\Sigma\to \R$ be a continuous function. The purpose of this section is to give sufficient conditions on the function $g$ to guarantee that the Lagrange spectrum of the data $(\sigma,\, g)$ contains ``significantly" portion of the image $g(\Sigma)$. The idea is to ensure that under some conditions on the function $g$, the set $g^{-1}(L(\sigma,\, g))$ has non-empty interior inside of some subshift of finite type.

We start describing the subset of sequences produced through elimination of a certain set of words from our collection of admissible words (we keep the notation introduced in the Subsection \ref{subsection:symbolic dynamics}). Given a set of finite words $\mathcal{W}$, we define the subset of $\Sigma$, $\Sigma(\mathcal{W}^c)$, of sequences which never enters, through shift action, into the cylinders determined by the words in $\mathcal{W}$. In other words,
\begin{align*}
    \Sigma({\mathcal{W}}^c) := \left\{
        \theta\in \Sigma \colon\,
        \sigma^n(\theta) \notin C_0(b),\,
        \forall b\in \mathcal{W},\, n\in \Z        
    \right\}.
\end{align*}
In this work we are particularly interested in the subset of sequences that are defined by exclusion of subwords of a given finite sequence $a\in \mathbb{A}^{2n+1}$. More precisely,
\begin{align*}
    \mathcal{W}_{1/3}(a) := \left\{
        b\in \cup_{j=1}^{|a|}\mathbb{A}^j \colon\,
        |b| \geq  \lfloor   |a|/3 \rfloor \text{ and } \exists\, j\in \Z \text{ with } C_j(a) \subset C_0(b)
    \right\}.
\end{align*}
Write $ \Sigma_{1/3}(a^c) :=  \Sigma(\mathcal{W}_{1/3}(a)^c)$. Observe that $\Sigma_{1/3}(a^c)\subset \Sigma$ is a shift invariant set and if the subshift $\Sigma$ is topologically mixing (have enough connections between the words), then it is reasonable to expect that, for words $a$ of sufficiently large length,  $(\sigma,\, \Sigma_{1/3}(a))$ is transitive. That is the content of the next proposition. 
\begin{proposition} \label{trans_sigma}
   Let $(\sigma,\, \Sigma)$ be a mixing subshift of finite. Then, there is a $n_0\in \N$ such that for every $n\geq n_0$ and every admissible word $a\in \mathbb{A}^{2n+1}$, $(\sigma,\, \Sigma_{1/3}(a^c))$ is a transitive subshift of finite type.
\end{proposition}
\begin{proof}
    Let $B$ be the aperiodic transitive matrix associated to $\Sigma$. Consider $m\in \N$ such that $B^m_{\alpha \beta}>0$ for every $\alpha, \beta \in \mathbb{A}$. Write $2n+1 = 3mk$, with the value of $k$ to be chosen later sufficiently large. So, given $a\in \mathbb{A}^{2n+1}$ admissible, any word $b\in \mathcal{W}_{1/3}(a)$ is a subword of $a$ with size at least $mk$.
    
    In order to prove that $(\sigma,\,  \Sigma_{1/3}(a^c))$ is transitive it is enough to guarantee that for any pair of admissible words $e,f\in \mathbb{A}^{mk}$ which is not in $\mathcal{W}_{1/3}(a)$ can be connected by concatenation of words outside of $\mathcal{W}_{1/3}(a)$.
    
    Write $e = (e_1,\ldots, e_{mk})$ and $f = (f_1,\ldots, f_{mk})$ and observe that for each finite word $b = (b_1,\ldots, b_{k-1})\in \mathbb{A}^{k-1}$ we can find a different gluing word connecting $e_{mk}$ to $f_1$ of length $mk$ taking $k$ gluing words of length $m$ connecting the $b_i$'s such as: $c_{e_{mk},b_1}, c_{b_1,b_2},\ldots, c_{b_{k-1}, f_1}$. Thus, the number of admissible word of length $mk$,  $B^{mk}_{e_{mk},f_1}$, satisfies
    \begin{align*}
    B^{mk}_{e_{mk},f_1} \geq \left(
        \# \mathbb{A}
    \right)^{k-1}.
    \end{align*}
    Since for $k$ sufficiently large $(\# A)^{k-1} > 6(mk)^2 > \# \mathcal{W}_{1/3}(a)$, we see that there exists an admissible connection (a gluing word of length $mk$) in $\Sigma_{1/3}(a^c)$ between the words $e$~and~$f$. 
\end{proof}

\begin{proposition}
\label{prop:trans_sigma2}
    Let $(\sigma,\, \Sigma)$ be a subshift of finite type which is a tower of size $k$ over a mixing subshift of finite type $\Sigma_0\subset \Sigma$. Then, there exists $n_0\in \N$ such that for every $n\geq n_0$ and every admissible word $a\in \mathbb{A}^{k(2n+1)}$, we have that $(\sigma,\,  \Sigma_{1/3}(a^c))$ is a transitive subshift of finite type.   
\end{proposition}
\begin{proof}
    Let $\mathbb{A}_0\subset \mathbb{A}^k$ such that $\Sigma_0 = \Sigma(\mathbb{A}_0)$ (see the notation introduced in Subsection \ref{subsubsection:symbolicTowers}). Since $(\sigma^k,\, \Sigma_0)$ is topologically mixing, by the Proposition \ref{trans_sigma},  there exists $n_0\in \N$ such that for every $n\geq n_0$ and every word $b\in \mathbb{A}_0^{2n+1}$ we have that $(\sigma^k, (\Sigma_0)_{1/3}(b^c))$ is a transitive subshift of finite type. For any $n\geq n_0$ and any $\Sigma$-admissible word $a\in \mathbb{A}^{k(2n+1)}$ we can associated a  $\Sigma_0$-admissible word $b \in \mathbb{A}_0^{2n+1}$ and for every $\theta\in \Sigma_{1/3}(a^c)$, we can find $0\leq j \leq k-1$ such that $\sigma^j(\theta)\in (\Sigma_0)_{1/3}(b^c)$. In other words, we can write $\Sigma_{1/3}(a^c) = \cup_{j=0}^{k-1}\, \sigma^{-j}(\Sigma_{1/3}(b^c))$. In particular, $(\sigma,\, \Sigma_{1/3}(a^c))$ is transitive.
\end{proof}

For each pair $(\alpha, \beta)\in \mathbb{A}^2$, fix $c_{\alpha,\beta}$ a gluing word of \emph{smallest} size connecting $\alpha$ and $\beta$. Consider a finite admissible word $a = (a_{-n},\ldots, a_0,\ldots, a_n)\in \mathbb{A}^{2n + 1}$ and define the map $H_a:\Sigma\to\Sigma$ given by $H_a(\theta) = (\theta^-,c_{\theta_0,a_{-n}}, a^*, c_{a_{n}, \theta_1}, \theta^+)$. More explicitly,
\begin{align*}
    H_a(\theta) := (\ldots \theta_{-2}, \theta_{-1}, \theta_0, c_{\theta_0,a_{-n}}, a_{-n}, \ldots, a_0^*, \ldots, a_{n}, c_{a_{n}, \theta_1}, \theta_1, \theta_2, \ldots).
\end{align*}
For later purposes it is important to highlight the map $H_a$ is an unstable holonomy map.
\begin{proposition}
\label{prop:boxMapHolonomy}
    For each $d\in \mathbb{A}^2$ $\Sigma$-admissible, there exists $k = k(a,d)\in \N$ such that for every $\xi\in C_0(d)$ and every $\theta\in C_0(d)\cap W^s_{\text{loc}}(\xi)$,
    \begin{align*}
        H_a(\theta) = \sigma^k \circ h^u_{\xi, \sigma^{-k} \circ H_a(\xi)}(\theta).
    \end{align*}
\end{proposition}
\begin{proof}
    Let $d = (d_0,d_1)\in \mathbb{A}^2$ be a finite word and observe that for every $\theta\in C_0(d)$ we can take $c_{\theta_0,a_{-n}} = c_{d_0,a_{-n}} =:e$ and $c_{a_{n},\theta_1} = c_{a_{n}, d_1} =: f$ and so for every $\theta\in C_0(d)$, $H_a(\theta) = (\theta^-,\, e,\, a^*,\, f,\, \theta^+)$. Now fix $\xi\in C_0(d)$ and take $\theta \in W^s_{\text{loc}}(\xi)\cap C_0(d)$, i.e., the sequence $\theta$ satisfies that $\xi_0 = \theta_0$ and $\xi^+ = \theta^+$. Then, taking $k := |e| + n$ and $\zeta := \sigma^{-k}(H_a(\xi))$, we have
\begin{align*}
    \sigma^{k}\circ h^u_{\xi, \zeta}(\theta)
    = \sigma^{k}(\theta^{-*},\zeta^+) 
    = \sigma^{k}(\theta^{-*}, e, a, f,\xi^+)
    = \sigma^{k}(\theta^{-*}, e, a, f, \theta^+)
    = H_a(\theta).
\end{align*} 
\end{proof}

The next proposition is the main result of this section and it is the mechanism that will allow us to prove the main result of this work.
\begin{proposition} \label{prop:150223.1}
    Let $g: \Sigma \to \R$ be a continuous function satisfying that 
    \begin{align} \label{SDH2}
        \sup_{\theta \in \Sigma\backslash C_{-n}(a)} g(\theta)
        < \inf_{\theta \in C_{-n}(a)} g(\theta),
    \end{align}
    for some admissible word $a\in \mathbb{A}^{2n+1}$ with $n > 3\, \displaystyle\max_{(\alpha,\beta)}\, |c_{\alpha,\beta}| + 1$. Then, there is an open set $\mathcal{O} \subset \Sigma_{1/3}(a^c)$ and $j_0 \in \N$ such that
    \begin{align*}
        \sigma^{j_0}\circ H_a(\mathcal{O})\subset C_{-n}(a)
        \quad
        \text{and}
        \quad
        g\circ\sigma^{j_0}\circ H_a(\mathcal{O})\subset L(\sigma, g).
    \end{align*}
\end{proposition}
\begin{proof}
    Write $a = (a_{-n},\ldots, a_0,\ldots, a_n)$ and  $d = (d_0,d_1)$ such that $C_0(d)\cap \Sigma_{1/3}(a^c) \neq$~$ \varnothing$. Also denote by $e, f$ the gluing words $c_{d_0,a_{-n}}, c_{a_{n}, d_1}$ of size $l$ and $m$, respectively. We define the auxiliary map $\widetilde{H}: C_0(d) \subset \Sigma \to \Sigma$ given by
    \begin{align*}
        \widetilde{H}(\theta)= (\ldots \theta_{-2}, \theta_{-1}, \theta_0, \tau_0^*, c_1, \tau_1, c_2, \ldots, c_k, \tau_k, c_{k+1}, \ldots),
    \end{align*}
    where $\tau_k = (\theta_{-k}, \ldots, \theta_0, e, a, f, \theta_1, \ldots \theta_{k+1})$, the $*$ indicates the zero-th position of $\widetilde{H}(\theta)$ located in $a_0$ which is the center of $a$ inside $\tau_0$ and  $c_k = c_{\theta_k, \theta_{-k}}$, $k\geq 1$. Observe that both $\tau_{k-1}$ and $c_k$ depend on $\theta$, for every $k \geq 1$.

    Let $l(k)$ be the position in $\widetilde{H}(\theta)$ of the center $a_0$ in $a$ in the center of $\tau_k$ and notice that for every $\theta = (\theta_n)_{n\in \Z}\in C_0(d)\cap \Sigma_{1/3}(a^c)$,
    \begin{align}\label{eq:070323.10}
        &[\sigma^{n + m}(\sigma^{l(k)}\circ \tilde{H}(\theta))]^+
        = (\theta_1, \theta_2, \ldots, \theta_{k+1}, c_{k+1}, \theta_{-k-1},\ldots)
        \quad
        \text{and}\nonumber\\
        &[\sigma^{- (n + l)}(\sigma^{l(k)}\circ\tilde{H}(\theta))]^-
        = (\ldots, \theta_k, c_k, \theta_{-k}, \ldots, \theta_{-1}, \theta_0).
    \end{align}
    In particular, if $j\in \{-2n - l,\ldots, 2n + m\}$, we have that
    \begin{align}\label{eq:070323.13}
        \lim_{k\to\infty}\sigma^j(\sigma^{l(k)}\circ \tilde{H}(\theta)) = \sigma^j\circ H_a(\theta).
    \end{align}
    Take $k > 4n$ and consider the intervals of integers
    \begin{align*}
        I_k^- := - (n + l) + [-k ,\, -n - 1]
        \quad
        \text{and}
        \quad
        I_k^+ := (n + m) + [n + 1,\, k+1 + |c_{k+1}|].
    \end{align*}
    Also write $I_k := [-(n + l) - k,\, (n + m) + k + 1 + |c_{k+1}|]$ and notice that $I_k^{\pm}\subset I_k$ and $I_k\backslash (I_k^-\cup I_k^+) = [-2n - l, \ldots, 2n + m]$ which is independent of $k$, see the Figure \ref{fig:sim}.
    \vspace{0.2cm}
     \begin{figure}[ht]
        \centering
        \includegraphics[scale=0.85]{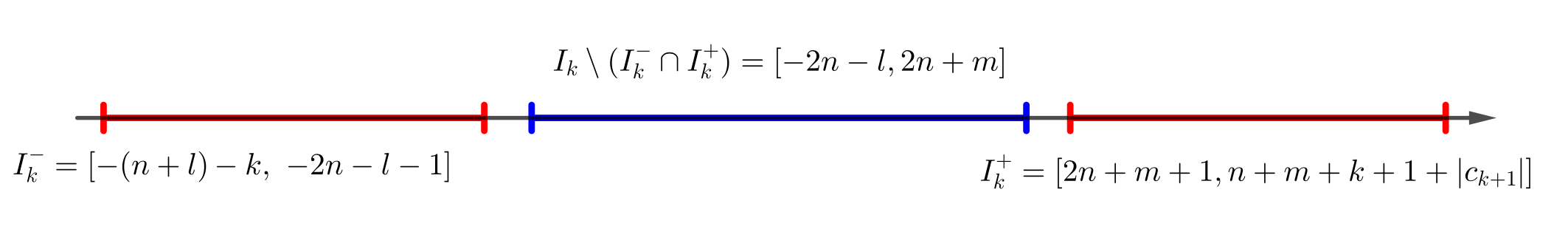}
        \caption{Interval of integers $I_n$}
        \label{fig:sim}
    \end{figure} 
    Moreover, we may decompose $\N$ as the disjoint union $\cup_{k\geq 1} l(k) + I_k$ and so,
    \begin{align}\label{eq:070323.11}
        \limsup_{j\to\infty} g(\sigma^j\circ\tilde{H}(\theta)) =
        \limsup_{k\to\infty}\sup_{j\in I_k}\, g(\sigma^{j+l(k)}(\tilde{H}(\theta))).
    \end{align}
    Take $j\in I_k^-\cup I_k^+$. Directly from the expressions in \eqref{eq:070323.10}, we can see that the interval of size $2n$ around the zero position of the sequence $\sigma^j(\sigma^{l(k)}\circ \tilde{H}(\theta))$ share a sub-word of size at least $\lfloor |a|/3 \rfloor$ with the word $(\theta_{-k},\ldots, \theta_0,\ldots, \theta_k)$ (here we are using the assumption $n > 3\, \max_{(\alpha,\beta)}\, |c_{\alpha,\beta}| + 1$). Since $\theta\in \Sigma_{1/3}(a^c)$, this implies that $\sigma^j(\sigma^{l(k)}\circ \tilde{H}(\theta)) \in \Lambda\backslash C_{-n}(a)$. Thus, by assumption in \eqref{SDH2},
    \begin{align*}
        \sup_{j\in I_k^-\cup I_k^+}\, g(\sigma^j(\sigma^{l(k)}\circ\tilde{H}(\theta))) < g(\sigma^{l(k)}\circ \tilde{H}(\theta))),
    \end{align*}
    which implies that
    \begin{align}\label{eq:070323.12}
        \sup_{j\in I_k}\, g(\sigma^j(\sigma^{l(k)}\circ\tilde{H}(\theta)))
        = \max_{j\in I_k\backslash (I_k^-\cup I_k^+)}\, g(\sigma^{j + l(k)}\circ\tilde{H}(\theta)).
    \end{align}
    Therefore, by equations \eqref{eq:070323.11} and \eqref{eq:070323.12}, there exists $j_0\in I_k\backslash (I_k^-\cup I_k^+)$ and a subset sequence $k_i\to\infty$ such that
    \begin{align}\label{eq:070323.14}
        \limsup_{i\to\infty} g(\sigma^i\circ \tilde{H}(\theta))
        = \lim_{k\to\infty} g(\sigma^{j_0 + l(k_i)}\circ \tilde{H}(\theta))
        = g( \sigma^{j_0}\circ H_a(\theta)),
    \end{align}
    where the last inequality is justified using equation \eqref{eq:070323.13}.
    For each $j\in I_k\backslash (I_k^-\cup I_k^+)$ consider the compact set
    \begin{align*}
        \Sigma_j = \left\{
            \theta \in \Sigma_{1/3}(a^c)\cap C_0(d)\colon\,
            \limsup_{k\to\infty} g(\sigma^k\circ \tilde{H}(\theta))
            = g( \sigma^{j}\circ H_a(\theta))
        \right\}.
    \end{align*}
    Observe that for every $\theta \in \Sigma_j$, $g(\sigma^j\circ H_a(\theta))\in L(\sigma, g)$. By equation \eqref{eq:070323.14}, we see that
    \begin{align*}
        \Sigma_{1/3}(a^c) \cap C_0(d)
        = \bigcup_{j\in I_k\backslash(I_k^-\cup I_k^+)}\, \Sigma_j.
    \end{align*}
    Then, by Baire's Theorem, there exists $j_0\in I_k\backslash (I_k^-\cup I_k^+)$ such that $\Sigma_{j_0}$ has no empty interior. So, there exists open set $\mathcal{O}\subset \Sigma_{j_0}\subset \Sigma_{1/3}(a^c)$ such that
    \begin{align*}
        g\circ\sigma^{j_0}\circ H_a(\mathcal{O})\subset L(\sigma,\, g).
    \end{align*}

    Finally, by definition of $\tilde{H}$, we have $\sigma^k(\tilde{H}(\theta))$ belongs to $a$ for infinitely many $k>0$. Thus, by \eqref{SDH2}, for every $\theta \in \Sigma_{1/3}(a^c)\cap C_0(d)$,
    \begin{align*}
        \displaystyle\limsup_{k\to\infty} g(\sigma^k(\tH(\theta))) > \sup_{\zeta\in \Lambda\backslash C_{-n}(a)} g(\zeta).
    \end{align*}
    However, for every $\theta \in \mathcal{O}\subset \Sigma_{j_0}$, $\displaystyle \limsup_{k\to\infty} g(\sigma^k\circ \tilde{H}(\theta)) = g( \sigma^{j_0}\circ H_a(\theta))$ which implies that $\sigma^{j_0}(H_a(\theta))\in C_{-n}(a)$.
\end{proof}
\section{Choice of the data set}
\label{section:choiceData}

In this section we provide the generic data set that will be used in the proof of Theorem \ref{thm:Main}. We start give conditions to produce the horseshoe with good geometric properties such as good regularity of the holonomy maps and the existence of regular Cantor sets inside of the horseshoe. Later we construct a set of $C^1$ real functions such that when restricted to the horseshoe satisfies the conditions to a large Lagrange spectrum presented in Proposition \ref{prop:150223.1}.

%%%%%%%%%%%%%%%%%%%%%%%%%%%%%%%%%%%%%%%%%%%%%%%%
\subsection{Choice of the horseshoe}
\label{subsec:choiceHorseshoe}
Let $\varphi: M\to M$ be a $C^r$-diffeomorphism with $r\geq 1$. Consider $p\in M$ a hyperbolic periodic point and denote by $E^s(p), E^u(p)$, respectively the stable and unstable space of $p$. We write $u = \dim E^u(p)$ and $s = \dim E^s(p)$, $u+s = d = \dim M$. Additionally, assume that $p$ admits a \emph{transversal homoclinic intersection} in a point $q_0$, meaning that the stable and unstable manifold of $p$ intersect each other transversely at $q_0\in M$. For simplicity, assume that $p$ is a fixed point of $\varphi$.

In what follows, we list conditions on the diffeomorphism $\varphi$ in order to obtain a geometrical horseshoe with intrinsic differentiable unstable holonomies.

\paragraph{Condition C1.} The fixed point $p$ has real and simple spectrum, i.e., the eigenvalues of $D\varphi(p)$ are real with distinct absolute values.

Write the spectrum of $D\varphi(p)$ as $\spec(D\varphi(p)) = \{\sigma_u,\ldots,\sigma_1,\lambda_1,\ldots,\lambda_s\}$, where $\sigma_j, \lambda_i \in \R$ for every $i=1,\ldots, u$, $j=1,\ldots, s$ and satisfies
\begin{align*}
    |\sigma_u|> |\sigma_{u-1}|>\ldots > |\sigma_1| > 1 > |\lambda_1|>|\lambda_2| > \ldots > |\lambda_s|.
\end{align*}
Consider the decomposition of $E^s(p)$ into $D\varphi$-invariant subspaces $E^w(p)\oplus E^{ss}(p)$, where $E^{w}(p)$ is the one dimensional eigenspace associated with the weakest contracting eigenvalue $\lambda_1$ and $E^{ss}(p)$ is the eigenspace associated with the strongest contracting eigenvalues $\lambda_2,\ldots, \lambda_s$. The spaces $E^{w}(p)$ and $E^{ss}(p)$ are called respectively the weak stable and the strong stable space associated with the point fixed $p$.

Using the Strong Stable Manifold Theorem \cite{Sh2013}, there exists a $\varphi$-invariant submanifold of the stable manifold of $p$, tangent to $E^{ss}(p)$ and denoted by $W^{ss}(p)$. This submanifold is called strong stable manifold at $p$.

\paragraph{Condition C2.} The intersection point $q_0$ does not belong to $W^{ss}(p)$.

\begin{remark}
\normalfont
  In the case that $q_0\in W^{ss}(p)$ we cannot expect in general a good regularity of the unstable lamination. Indeed, there are simple examples of diffeomorphisms $\varphi$ with such property in which the unstable holonomies are not even Lipschitz. See \cite[Exemple 3.1]{PaVi1994}.
\end{remark}

\paragraph{Condition C3.} Assume that $\varphi$ is linear around $p$. More specifically, there are $C^3$-coordinates in a neighborhood $U_0$ of $p$, say $(\xi, \zeta_1, \zeta^{ss}) = (\xi_1,\ldots,\xi_u,\zeta_1,\ldots,\zeta_s)\in \R^u\times\R\times\R^{s-1}$, such that
\begin{enumerate}
    \item $E^u(p) = [\zeta=0]$, $E^s(p) = [\xi = 0]$, $E^w(p) = [\xi = \zeta^{ss} = 0]$ and $E^{ss} = [\xi = \zeta_1 = 0]$;
    
    \item $T := D\varphi(p)$ is diagonal in blocks with respect to the decomposition $\R^u\times\R\times\R^{s-1}$;
    
    \item In these coordinates, $\varphi = T$ in $U_0$.
\end{enumerate}

Fix $N_0\in \N$ such that $q_1 := \varphi^{-N_0}(q_0)\in U_0$.
\paragraph{Condition C4.} Assume that
\begin{align}
\label{eq:coneField+regularity}
    D\varphi^{N_0}(q_1)\left( E^u(p)\oplus E^w(p)\right)\, \cap\, E^{ss}(p) = \{0\}.
\end{align}

\begin{proposition}
\label{prop:genericityConditionsC1-C4}
    Among the $C^{\infty}$-diffeomorphisms with a hyperbolic periodic point admitting a transverse homoclinic intersection, there exists a $C^1$-open and $C^{\infty}$-dense set of diffeomorphisms satisfying the conditions C1-C4 above.
\end{proposition}
\begin{proof}
    The possibility to perturb a diffeomorphism in the $C^{\infty}$-topology to guarantee condition C1 was established in \cite{AvCrWi2021} (see \cite{BeMo2023} for a more elementary proof). The condition C2 is achieved through a $C^{\infty}$ local perturbation around the transversal homoclinic point $q_0$. Notice that conditions C1 and C2 are $C^1$-open inside $\diff^{\infty}(M)$.

    We denote by $\mathcal{T}_{1,2}^{\infty}(M)$ the set of diffeomorphisms with a hyperbolic periodic point of saddle-type admitting a transverse homoclinic intersection satisfying the conditions C1 and C2. Using Proposition \ref{prop:DiffResonantCondition} we can see that there exists a subset $\mathcal{L}^{\infty}(M)$ which is $C^1$-open and $C^{\infty}$-dense in the set of diffeomorphisms admitting a hyperbolic periodic point of saddle-type $H^{\infty}(M)$, such that if $\varphi\in \mathcal{L}^{\infty}(M)\cap \mathcal{T}_{1,2}^{\infty}(M)$, then $\varphi$ satisfies condition C3. It is also direct that condition C4 can be obtained for a $C^1$-open and $C^{\infty}$-dense subset of $\mathcal{L}^{\infty}(M)\cap \mathcal{T}_{1,2}^{\infty}(M)$. Taking the intersection of  these sets above described satisfying conditions C1-C4, we have the result.
\end{proof}

Let $\varphi:M\to M$ be a $C^r$-diffeomorphism, $r\geq 2$, satisfying the conditions C1-C4. Consider $N_0\in \N$ such that $q_1=\varphi^{-N_0}(q_0)\in W^u_{\eta_0}(p)$ for some $\eta_0$ small enough satisfying that $W^u_{\eta_0}(p)$ is contained in the neighborhood $U_0$ of $p$ given by condition C3.

For each $\delta >0$, consider the following compact neighborhood of $p$ and~$q_0$, 
\begin{align*}
    V_{\delta} = \left\{
        (\xi,\zeta)\in \R^d\colon\,
        \norm{\xi}\leq \delta,\, \norm{\zeta}\leq \rho
    \right\}\subset U_0,
\end{align*}
where $\rho>0$ is chosen (and then fixed) in such a way that $q_0\in V_{\delta}$. Set $n_0 = n_0(\delta) = \inf\left\{n\geq 1\colon\, q_1\in {\rm int } \,\varphi^n(V_{\delta})\right\}$.

Notice that $n_0$ increases as $\delta$ goes to zero. If $\psi := \varphi^{N_0 + n_0}$, then, for $\delta$ small enough, $V_\delta\cap \psi(V_{\delta})$ has two connected components, one containing $p$ and the other containing $q_0$, respectively denoted by $V_1(\delta)$ and $V_2(\delta)$. Consider the $\psi$-maximal invariant set inside of $V_{\delta}$, i.e.,
\begin{align*}
    \Lambda_0 = \Lambda_0(\delta) := \bigcap_{j\in \Z}\, \psi^j\left( V_{\delta} \right).
\end{align*}
This is a hyperbolic, $\psi$-invariant set and $\varphi^{N_0+n_0}|_{\Lambda_0}$ is conjugated to a full shift of two symbols $\sigma:\{0,1\}^{\Z}\to \{0,1\}^{\Z}$. We also denote by
\begin{align*}
    \Lambda = \bigcup_{j=0}^{N_0 + n_0 -1}\, \varphi^{-j}\left(\Lambda_0\right),
\end{align*}
the $\varphi$-invariant hyperbolic set (conjugated to a subshift of finite type) defined through $\Lambda_0$. The horseshoe $\Lambda$ has a symbolic structure as described in Proposition \ref{prop:symbolicStructureForTowerHorseshoe}. We denote by $E^s$ and $E^u$ the stable and unstable bundle of $\varphi$ at $\Lambda$ (for more details about the classical construction of the geometric horseshoe as above see~\cite[Section 7.4]{Ro1998}).

In the next proposition we gather a few properties of the hyperbolic set $\Lambda$ which will be essential in this work.
\begin{proposition}
\label{prop:niceHorseshoeProperties}
    There exist $\delta_0>0$ and $\alpha \in (0,1]$ uniform in a $C^2$-neighborhood of $\varphi$ such that for every $\delta\in (0,\delta_0)$, the above constructed horseshoe  $\Lambda$ satisfies the following:
    \begin{enumerate}
        \item \label{prop:horseshoeProperties-item1} The stable bundle $E^s$ of $\Lambda$ admits a dominated splitting of the form $E^s = E^w\oplus E^{ss}$, with $\dim E^w= 1$.

        \item \label{prop:horseshoeProperties-item2} There exist $C>0$ and a $(C,\alpha)$-H\"older continuous map $A:W^u_{loc}(\Lambda)\times W^u_{loc}(\Lambda)\rightarrow \M_{s-1, u+1}(\R)$ such that
        \begin{align*}
            \norm{A(x,z)}\leq C
            \quand
            (x-z)_{ss} = A(x,z)(x-z)_{uw},
        \end{align*}
        for all $x,z\in W^u_{loc}(\Lambda)$. Moreover, $A(p,p) = 0$.

        \item \label{prop:horseshoeProperties-item3} For every $x,y\in \Lambda$, $y\in W^u_{\delta}(x)$ the holonomies $H^u_{x,y} : W^s_{\delta}(x) \cap \Lambda\to W^s_{\delta}(y) \cap \Lambda$ are $IC^{1+\alpha}$. Moreover, $\Delta H_{x,y}^u(x,x)\, E^w(x) = E^w(y)$.
    \end{enumerate}
\end{proposition}
\begin{proof}
    To prove the item \ref{prop:horseshoeProperties-item1} we will use the cone field criterion (see~\cite[Theorem 2.6]{CrPo2015}) building first an $\psi$-invariant cone over $(R_1(\delta)\cup R_2(\delta))\cap V_{\delta}\cap \psi(V_{\delta})$, where $R_i(\delta) := \psi^{-1}(V_i(\delta))$, $i=1,2$. 
    
    Notice that for every $x\in R_1(\delta)\cap V_{\delta}\cap \psi(V_{\delta})$, $\varphi^j(x)\in U_0$, for every $0\leq j\leq N_0+n_0$. This implies in particular that for every $x\in R_1(\delta)\cap V_{\delta}\cap \psi(V_{\delta})$, $D\psi(x) = T^{N_0+n_0}\, x$. Also observe that $\varphi^{-N_0}(V_2(\delta)) = \varphi^{n_0}(R_2(\delta))$ is a small neighborhood of $q_1$ whose diameter decreases to $0$ when $\delta$ goes to $0$. Moreover, for every $x\in R_2(\delta)\cap V_{\delta}\cap \psi(V_{\delta})$ and every $0\leq j \leq n_0$, $\varphi^j(x)\in U_0$ and so $D\psi(x) = D\varphi^{N_0}(y)\, T^{n_0}\, x$, where $y = \varphi^{n_0}(x)$ belongs to the small neighborhood $\varphi^{-N_0}(V_2(\delta))$ of $q_1$.

    By condition C4, decreasing $\delta$ if necessary, there exists $\varepsilon>0$ such that, for every $x\in \varphi^{-N_0}(V_2(\delta))$
    \begin{align}\label{eq:angleDistanceConeField}
        \measuredangle\left(
            D\varphi^{N_0}(x) (E^u(p)\oplus E^w(p)),\, E^{ss}(p)
        \right) > \varepsilon > 0.
    \end{align}

    For each number $\gamma\in [0, \pi/2]$, denote by $C_{u,w}(\gamma)\subset \R^d$ the constant cone field over $(R_1(\delta)\cup R_2(\delta)) \cap V_{\delta}\cap \psi(V_{\delta})$ given by the set of directions $\R^d$ whose angle with respect to $E^u(p)\oplus E^w(p)$ is at most $\gamma$. Let $\tau>0$ such that for every $x\in \varphi^{-N_0}(V_2(\delta))$,
    \begin{align}\label{eq:DphiConeAction}
        D\varphi^{N_0}(x)\, C_{u,w}(\tau) \subset \interior C_{u,w}\left(
            \pi/2 - \varepsilon/2
        \right).
    \end{align}
    Notice that this is achievable using the inequality \eqref{eq:angleDistanceConeField}. Diminishing $\delta$ once more, we may also assume that
    \begin{align}\label{eq:TConeContraction}
        T^{n_0}\, C_{u,w}(\pi/2-\varepsilon/2) \subset \interior C_{u,w}(\tau).
    \end{align}
    This is due to the fact that any direction which is not inside of $E^{ss}(p)$ converges to $E^u(p)\oplus E^w(p)$ through $T$ iterations and so the cone $C_{u,w}(\pi/2-\varepsilon/2)$ will be contracted in the $E^u(p)\oplus E^w(p)$ direction.

    For any $y\in R_1(\delta)\cap V_{\delta}\cap \psi(V_{\delta})$ and $0\leq j \leq N_0+n_0$, $\varphi^j(y)\in U_0$ and so,
    \begin{align*}
        D\psi(y)\, C_{u,w}(\pi/2-\varepsilon/2) \subset \interior C_{u,w}(\pi/2-\varepsilon/2).
    \end{align*}
    Now, for $y\in R_2(\delta)\cap V_{\delta}\cap \psi(V_{\delta})$, $\varphi^{n_0}(y)\in \varphi^{-N_0}(V_2(\delta))$ and so by equations \eqref{eq:DphiConeAction} and \eqref{eq:TConeContraction}, we have
    \begin{align*}
        D\psi(y)\, C_{u,w}(\pi/2-\varepsilon/2)
        &= D\varphi^{N_0}(\varphi^{n_0}(y))\, T^{n_0}\, C_{u,w}(\pi/2-\varepsilon/2)\\
        &\subset \interior D\varphi^{N_0}(\varphi^{n_0}(y))\, C_{u,w}(\tau)\\
        &\subset \interior C_{u,w}(\pi/2-\varepsilon/2).
    \end{align*}
    Thus the constant cone field $C_{u,w}(\pi/2-\varepsilon/2)$ over $(R_1(\delta)\cup R_2(\delta))\cap V_{\delta}\cap \psi(V_{\delta})$ is invariant by the action of $D\psi$.

    A similar argument shows that the complementary cone field of $C_{u,w}(\pi/2-\varepsilon/2)$, namely $C_{ss}(\varepsilon/2)$, defined as the set of directions in $\R^d$ whose angle with the space $E^{ss}$ does not exceed $\varepsilon/2$, is invariant by the backwards action $D\psi^{-1}$ over $V_1(\delta)\cup V_2(\delta)$. However, in this case we need an additional modification on $\delta_0$ guaranteeing
    \begin{align*}
        T^{-n_0}\, C_{ss}(\pi/2 - \tau) \subset \interior C_{ss}(\varepsilon/2),
    \end{align*}
    where $C_{ss}(\pi/2 - \tau)$ is defined analogously as before.

    Thus, over $\cup_{i,j}V_i(\delta)\cap R_j(\delta)$ we have a pair of cone fields $C_{u,w}(\pi/2-\varepsilon/2)$ and $C_{ss}(\varepsilon/2)$ respectively invariant by $D\psi$ and $D\psi^{-1}$. Therefore, by the cone field criterion we have a dominated splitting over $\Lambda_0 = \cap_{k\in \Z} \psi^k(\cup_{i,j}V_i(\delta)\cap R_j(\delta))$ given by $T\Lambda_0 = F\oplus E$, where for every $x\in \Lambda_0$,
    \begin{align*}
        F(x) = \bigcap_{j\geq 0}D\psi^j(\psi^{-j}(x))\, C_{uw}(\pi/2-\varepsilon/2)
        \quad
        \text{and}
        \quad
        E(x) = \bigcap_{j\geq 0}D\psi^{-j}(\psi^j(x))\, C_{ss}(\varepsilon/2).
    \end{align*}
    Notice that $F(p) = E^u(p)\oplus E^w(p)$ and $E(p) = E^{ss}(p)$. Also, $\dim F = u + 1$ and $\dim E = s-1$. Moreover, $E\subset E^s$. Set
    \begin{align*}
        E^w := F\cap E^s
        \quad
        E^{ss} := E \subset E^s.
    \end{align*}
    These are $D\psi$-invariant continuous bundles and they give the dominated splitting $E^u\oplus E^w \oplus E^{ss}$ of $T\Lambda_0$. The dominated splitting for $\Lambda$ is obtained through iterations by $\varphi$ of the subspaces on the above decomposition.   This concludes the proof of item \ref{prop:horseshoeProperties-item1}.

    For a proof of item \ref{prop:horseshoeProperties-item2} see \cite[Proposition 3.2]{PaVi1994}. The regularity of the holonomies $H^u_{x,y}:W^s_{\delta}(x)\cap \Lambda\to W^u_{\delta}(y)\cap \Lambda$ expressed in item \ref{prop:horseshoeProperties-item3} is obtained in \cite[Proposition 3.5]{PaVi1994}. This proof also gives that there exists $C>0$, uniform for $\Delta H^u_{x,y}$ in a compact neighborhood, such that
    \begin{align}
        \label{eq:distanceToIdentityIDH}
            \sup_{z\in W^s_{\delta}(x)\cap \Lambda}\norm{
                \Delta H^u_{x,y}(z,z) - I 
            }\leq C\, d(x,y).
    \end{align}
    We briefly sketch the required adaptations in the argument from \cite{PaVi1994} to obtain inequality \eqref{eq:distanceToIdentityIDH}. Following carefully the steps of the proof of \cite[Proposition 3.5]{PaVi1994} we see that after choosing appropriate coordinates, we can choose a intrinsic derivative to $H^u_{x,y}$, in any pair of points $z_1,z_2\in W^s_{\delta}(x)\cap \Lambda$,  to be a solution $\Gamma(t;z_1,z_2)$ of a initial value problem
     \begin{equation*}
        \begin{cases}
         \frac{d}{dt}\Gamma(t;z_1,z_2) = \Delta H_t(g(t;z_1), g(t;z_2))\Gamma(t;z_1,z_2) \\
         \Gamma(0,z_1,z_2) = I.
        \end{cases}
    \end{equation*}
    (keeping the same notation as used in \cite{PaVi1994}) at the time $\xi = d(x,y)$ (in \cite{PaVi1994} the coordinates are chosen such that the transversal sections have distance one, but we could choose such coordinates preserving the distance between the sections inside the manifold $M$). Therefore, the inequality \eqref{eq:distanceToIdentityIDH} follows from compacity and the mean value theorem applied to the difference $\norm{\Gamma(\xi;z,z) - \Gamma(0;z,z)}$, for ${z\in W^s_{\delta}(x)\cap \Lambda}$.
    
    Now we prove that the intrinsic derivative of the holonomy preserves the weak stable bundle. For every $n\geq 1$, write $H_n := H^u_{\varphi^{-n}(x),\, \varphi^{-n}(y)}$ for the unstable holonomy connecting $\varphi^{-n}(x)$ to $\varphi^{-n}(y)$. By the invariance of the holonomy \eqref{eq:holonomyInvariance}, we may write, for every $n\geq 1$, $H := H^u_{x,y} = \varphi^n\circ H_n\circ \varphi^{-n}$. Thus,
    \begin{align}
    \label{eq:holPreservingWeak0}
        \Delta H(x,x) = \Delta\varphi^n\left(
            H_n(\varphi^{-n}(x)), H_n(\varphi^{-n}(x))
        \right)\,
        \Delta H_n(\varphi^{-n}(x),\varphi^{-n}(x))\,
        \Delta\varphi^{-n}(x,x).
    \end{align}
    Consider $v\in E^w(x)$ and denote by $v_n = \Delta\varphi^{-n}(x,x)\, v$ and $w_n = \Delta H_n(\varphi^{-n}(x),\varphi^{-n}(x))\, v_n$. Notice that $v_n\in E^w(\varphi^{-n}(x))\subset E^s(\varphi^{-n}(x))$ and $w_n\in E^s(\varphi^{-n}(y))$. Since $y\in W^u_{\delta}(x)$, we have that $d(\varphi^{-n}(x),\, \varphi^{-n}(y))\to 0$ as $n\to \infty$ and so, 
        \begin{align}\label{eq:holPreservingWeak1}
            d(E^{ss}(\varphi^{-n}(y)),\, E^{ss}(\varphi^{-n}(x))) \to 0.
        \end{align}
    Furthermore, by inequality (\ref{eq:distanceToIdentityIDH}), there exists a constant $C>0$ such that
    \begin{align} \label{eq:holPreservingWeak3}
        \norm{
                \Delta H_n(\varphi^{-n}(x),\,\varphi^{-n}(x)) - I
            } \leq\, C\, d(\varphi^{-n}(x),\, \varphi^{-n}(y)).
    \end{align}
    Using the fact that the splitting $E^s = E^w\oplus E^{ss}$ is dominated, we have that there exists $C_1>0$ such that
    \begin{align}\label{eq:holPreservingWeak2}
        \measuredangle\left(
            E^w(\varphi^{-n}(x)),\, E^{ss}(\varphi^{-n}(x))
        \right)\geq C_1 > 0,
    \end{align}
    for every $n\geq 0$. Hence, by \eqref{eq:holPreservingWeak1}, \eqref{eq:holPreservingWeak3} and \eqref{eq:holPreservingWeak2}, we have there exist $C_2>0$ and $n_1\in\N$ such that for every $n\geq n_1$, we have
    \begin{align*}
        \measuredangle\left(
            w_n,\,
            E^{ss}(\varphi^{-n}(y))
        \right)\geq C_2 > 0.
    \end{align*}    
    Since $E^w$ is an attractor to $D\varphi$-action, we obtain that
    \begin{align*}
        \measuredangle\left(
            \Delta\varphi^n\left(
            H_n(\varphi^{-n}(x)), H_n(\varphi^{-n}(x))
        \right)\, w_n,\,
        E^w(y)
        \right)\to 0,
    \end{align*}
    as $n\to \infty$. Therefore, by equation \eqref{eq:holPreservingWeak0}, $\Delta H(x,x)\, v\in E^w(y)$. This finishes the proof of the item \ref{prop:horseshoeProperties-item3} and the proposition.
\end{proof}

%%%%%%%%%%%%%%%%%%%%%%%%%%%%%%%%%%%%%%%%%%%%%%%%%%%%
\subsection{Construction of the regular Cantor set}
\label{subsection:cantor}

Let $\tLambda$ be a subhorseshoe of $T\Lambda$ containing $p$. Observe that $\tilde{\Lambda}$ has a dominated splitting of the form $E^u\oplus E^w\oplus E^{ss}$ which is given by the restriction of the splitting on $T\Lambda$. Moreover, the holonomies on $\tilde{\Lambda}$, being the restriction of the holonomies on $\Lambda$, are also intrinsically $C^{1+\alpha}$. We also may restrict the map $A: W^u(\Lambda)\times W^u(\Lambda)\to \M_{s-1, u+1}(\R)$ to $W^u(\tilde{\Lambda})\times W^u(\tilde{\Lambda})$ keeping the same properties as described in item \ref{prop:horseshoeProperties-item2} of Proposition \ref{prop:niceHorseshoeProperties}.

Consider the $\psi$-invariant set, $\tilde{\Lambda}_0 := \tilde{\Lambda}\cap \Lambda_0$ and notice that $\tilde{\Lambda} = \cup_{j=0}^{N_0+n_0 - 1}\, \varphi^{-j}(\tilde{\Lambda}_0)$. Let $\tilde{V}_0$ be a small neighborhood of $p$ where we have local product structure for $\psi$. In particular, there exists $\eta_1>0$ small such that for every $x\in \tilde{\Lambda}_0\cap\tilde{V}_0$ we have that $W^u_{\eta_1}(x)\cap W_{\eta_1}^s(p)$ is a single point contained in $\tilde{\Lambda}_0$. Define the projection $\pi_w:\tilde{\Lambda}_0 \cap \tilde{V}_0\cap W^s_{\text{loc}}(p)\to \R$ given by
\begin{align*}
    \pi_w(\xi_u,\ldots, \xi_1, \zeta_1,\ldots, \zeta_s) = \zeta_1.
\end{align*}
Notice that $\pi_w$ is a homeomorphism with its image $K^w:= \pi_w(\tilde{\Lambda}_0\cap \tilde{V}_0\cap W^s_{\text{loc}}(p))$. Indeed, by item \ref{prop:horseshoeProperties-item2} of Proposition \ref{prop:niceHorseshoeProperties}, for every $x,y\in \tilde{\Lambda}_0\cap \tilde{V}_0\cap W^s_{\text{loc}}(p)$,
\begin{align*}
    \norm{x-y} 
    &= \norm{
        (0_u, (x-y)_w, (x-y)_{ss})
    }\\
    &= \norm{
        (0_u, \pi_w(x) - \pi_w(y), A(x,y)\, (0_u, \pi_w(x) - \pi_w(y)))
    }\\
    &\leq (1+C)|\pi_w(x) - \pi_w(y)|.
\end{align*}
In particular, this shows that $\pi_w^{-1}$ is a Lipschitz function from $K^w$ to $\tilde{\Lambda}_0\cap  \tilde{V}_0\cap W^s_{\text{loc}}(p)$. Furthermore, we can define an intrinsic derivative to $\pi_w^{-1}$ in the following way: for any \linebreak $t,s \in K^w$ set
\begin{align}\label{eq:IDerivativeWeakPiInverse}
    \Delta\pi_w^{-1}(t,s):= \left(
        1_w,\, A(\pi_w^{-1}(t),\, \pi_w^{-1}(s))\, (0_u,\, 1_w)
    \right) \in \M_{s,w}(\R).
\end{align}
Thus,
\begin{align*}
    \Delta\pi_w^{-1}(t,s)\, (t-s)
    &=  \left(
        t-s,\, A(\pi_w^{-1}(t),\, \pi_w^{-1}(s))\, (0_u,\, t-s)
    \right)\\
    &= \left(
        t-s, A(\pi_w^{-1}(t),\, \pi_w^{-1}(s))\, (\pi_w^{-1}(t) - \pi_w^{-1}(s))_{uw}
    \right)\\
    &= (t-s, (\pi_w^{-1}(t) - \pi_w^{-1}(s))_{ss})
    = \pi_w^{-1}(t) - \pi_w^{-1}(s).
\end{align*}

Since $\tLambda_0$ is a horseshoe for $\psi$, using an appropriate coding for $\psi|_{\tLambda_0}$, we identify the elements of $\tLambda_0$ and $\psi|_{\tLambda_0}$, respectively,  with sequences $(\theta_n)_{n\in \Z}\in \mathbb{A}^{\Z}$, for some finite alphabet $\mathbb{A}$, and the shift map with the set of admissible sequences (see Subsection \ref{subsection:symbolic dynamics}). For convenience, we represent the fixed point $p$ using the symbolic notation writing $p = (\bar{0})$. When describing a point $x\in \tLambda_0$ as a concatenation of finite words on $\mathbb{A}$,
\begin{align*}
    x = (\ldots, b_{-n+1},\ldots, b_{-1}, b_0^{\ast}, b_1,\ldots, b_{n-1},\ldots),
\end{align*}
the $\ast$ indicates the position of the zero-th coordinate of the sequence. In the above example, the zero-th coordinate of the point $x$ is located in one of the coordinates of the finite word $b_0\in \mathbb{A}^m$. We also use the notation $0^{k} := (0,\ldots, 0)\in \mathbb{A}^{k}$, for every $k\in \N$ and $\bar{0} = (0,0,\ldots)\in \mathbb{A}^{\N}$.  

Using this symbolic notation, we can write $\tLambda_0\cap \tilde{V}_0 = C_{-m+1}(0^{2m-1})$, with $m\in \N$. Take $l>2m$ and set 
\begin{align*}
    \hat{W}_l := \bigcup_{b:\, |b|=l-2m+1}\, \hat{W}_{b,l},
\end{align*}
where the union is being taken in finite ($\tLambda_0$-admissible for $\psi$) words $b\in \mathbb{A}^{l-2m+1}$ and
\begin{align*}
    \hat{W}_{b,l} := \bigcap_{j\geq 0}\, \psi^{jl}\left(
        C_{-m+1}(0^{2m-1})
    \right)\cap C_m(b\, \bar{0})
    \subset \tLambda_0
\end{align*}
So, for each $x\in \hat{W}_{b,l}$, there exist a sequence $(c_{-i})_{i> 0}\subset \mathbb{A}^{l-2m+1}$ such that
\begin{align*}
    x = (\ldots, c_{-i+1},\, 0^{2m-1},\, c_{-i+2}, \ldots, \, {0}^{2m-1}, \, c_{-2},\, {0}^{2m-1}, \, c_{-1},\, 0^{2m-1\ast},\, b,\, \bar{0}),
\end{align*}
where the $\ast$ is on the $0$ in the center of the word $0^{2m-1}$. Write $K^w_l := \pi_w\circ \psi^l(\hat{W}_l) = \cup_b\, K^w_{b,l}$, where $K^w_{b,l} := \pi_w\circ\psi^l(\hat{W}_{b,l})$.

\begin{proposition}
\label{prop:regularCantor}
    For every sufficiently large $l > 2m$, $K^w_l$ is a regular Cantor set. Moreover,
    \begin{align*}
        IT_0\, K^w_l =  \Delta\, \pi_w(0,0)\cdot E^w(p).
    \end{align*}
\end{proposition}
\begin{proof}
    Observe that, by the local product structure of $\tilde{V}_0\cap \tilde{\Lambda}_0$, $W^u_{\eta_1}(x)\cap W^s_{\text{loc}}(p)\subset \tLambda_0$ is a singleton, which we denote by $\{\pi_s(x)\}$. Using \cite[Proposition 4.3]{PaVi1994}, we see that $\pi_s:\tLambda_0\cap \tilde{V}_0\to W^s_{\text{loc}}(p)\cap\tLambda_0$ is an intrinsic $C^{1+\alpha}$ map.

    Now we define the $IC^{1+\alpha}$ map $g_l:K^w_l\to K^w_l$ by $g_l := \pi_w\circ \pi_s\circ \psi^{-l}\circ \pi_w^{-1}$ which gives the regular Cantor set structure on $K^w_l$. First, we need to guarantee that this is a well defined map, meaning that for every $t\in K^w_l$, $g_l(t)\in K^w_l$. Indeed, by definition $K^w_l$, $t = \pi_w\circ \psi^l(x)$ with 
    \begin{align*}
        x = (\ldots, c_{-i+1},\, 0^{2m-1},\, c_{-i+2}, \ldots, c_{-1},\, 0^{2m-1\ast},\, c_0,\, \bar{0}),
\end{align*}
    for a sequence of finite words $(c_{-i})_{i\geq 0} \in \mathbb{A}^{l-2m+1}$. Notice that by the local product structure we can also characterize $\pi_s(x)$ symbolically, 
    \begin{align*}
        \pi_s(x) = (\ldots, c_{-i+1},\, 0^{2m-1},\, c_{-i+2}, \ldots, c_{-1},\, 0^{2m-1\ast},\, \bar{0}) \in W^u_{\eta_1}(x)\cap W^s_{\text{loc}}(p).
    \end{align*}
    In particular, we can write
    \begin{align*}
        \pi_s(x) = \psi^l(\ldots, c_{-i},\, 0^{2m-1},\, c_{-i+1}, \ldots, c_{-2},\, 0^{2m-1\ast},\, c_{-1},\, 0^{2m-1},\,  \bar{0}),
    \end{align*}
    implying that $\pi_s(x) \in \psi^l(\hat{W}_{c_{-1}, l})$. So, $g_l(t) = \pi_w\circ \pi_s(x) \in K^w_{c_{-1}, l}\subset K^w_l$. Thus, the $g_l$ is a well defined map. 

    The restrictions $g_{b,l} = g_l|_{K^w_{b,l}}: K^w_{b,l}\to K_l^w$, for $b\in \mathbb{A}^{l-2m+1}$ admissible, are homeomorphisms. From the above computation, the injectivity is clear. To show that $g_{b,l}$ is onto, consider $s \in K^w_l$, then we can write
    \begin{align*}
        \pi_w^{-1}(s) &= \psi^l(\ldots, c_{-i+1},\, 0^{2m-1},\, c_{-i+2}, \ldots, c_{-1},\, 0^{2m-1\ast},\, c_0,\, \bar{0})\\
        &= (\ldots, c_{-i},\, 0^{2m-1},\, c_{-i+1}, \ldots, c_{0},\, 0^{2m-1\ast},\, \bar{0}) \in \psi^l(\hat{W}_l),
    \end{align*}
    for some sequence $(c_{-i})_{i\geq 0}$. Write $t = \pi_w\circ \psi^l(x)\in K^w_{b,l}$, where $x\in \hat{W}_{b,l}$ satisfies
    \begin{align*}
        x = (\ldots, c_{-i},\, 0^{2m-1},\, c_{-i+1}, \ldots, c_{0},\, 0^{2m-1\ast},\, b,\, \bar{0}).
    \end{align*}
    Then $\pi_s(x) = \pi_w^{-1}(s)$ and so $g_{b,l}(t) = \pi_w\circ\pi_s(x) = s$. The fact that, for each $b$, $g_{b,l}$ is bijective and continuous defined in a compact is enough to guarantee that $g_{b,l}$ is a homeomorphism.

    Furthermore, since $\pi_s$ is $IC^{1+\alpha}$, we have that for each $b$ admissible and $l$ sufficiently large, $g_{b,l}$ is a $IC^{1+\alpha}$ map satisfying that
    for every $t,s\in K^w_{b,l}$, $\norm{\Delta g_{b,l}(s,t)} \geq 2$. Moreover, $\text{Hull}(K^w_{b,l})\cap \text{Hull}(K^w_{b',l}) = \varnothing$, for every $b\neq b'$. In particular, using the extension result \cite[Lemma 4.4]{PaVi1994}, we see that $K^w_l$ is a $C^{1+\alpha}$ regular Cantor set.

    To finish the proof of the proposition, observe that by item \ref{prop:horseshoeProperties-item2} of Proposition \ref{prop:niceHorseshoeProperties}, we have that $A(p, p) = 0$ and so by equation \eqref{eq:IDerivativeWeakPiInverse},
    \begin{align*}
        \Delta \pi_w^{-1}(0,0) = (1_w,\, A(p,p)(0_u,1_w)) = (1_w, 0_{ss})\in E^w(p).
    \end{align*}
    Therefore, $\Delta \pi_w^{-1}(0,0)(IT_0\, K^w_l) = \text{span}\{(1_w,0_{ss})\} = E^w(p)$.
\end{proof}

%%%%%%%%%%%%%%%%%%%%%%%%%%%%%%%%%%%%%%%%%%%%%%%%%%%
\subsection{Choice of the potential}
\label{section:genericPotentials}
In this section, we construct the large set of potentials in $C^1(M;\mathbb{R})$ that will be used in the main theorem as functions in the definition of the dynamically defined Lagrange spectrum. We also deal with a general horseshoe $\Lambda$ for a diffeomorphism $\varphi\in \diff^r(M)$ containing a point $p\in\Lambda$ which is fixed by $\varphi$.

Let $\mathcal{P} = \{P(\alpha);\ \alpha\in \mathbb{A}\}$ be a Markov partition for the horseshoe $\Lambda$ associated with $\varphi$, and let $B=(B_{\alpha, \beta})_{\alpha, \beta \in \mathbb{A}}$ be the related transition matrix (see Subsection \ref{subsection:symbolic dynamics}).  For each finite word $a = (a_1,\ldots, a_{2n+1})\in \mathbb{A}^{2n+1}$, we consider the box
\begin{align*}
    P(a) := \bigcap_{i=-n}^{n}\varphi^{-i}(P(a_{i+n+1})).
\end{align*}
The collection of the sets $P(a)$ when $a$ varies over the set of admissible words in $\mathbb{A}^{2n+1}$ will be denoted by $\mathcal{P}^n$. Notice that $\mathcal{P}^n$ also forms a Markov partition for $\Lambda$.

For each $x\in \Lambda$ fix one unit vector $e^w(x)$ in $E^w(x)$ in such a way that $\Lambda\ni x\mapsto e^w(x)$ is continuous. Let $\mathcal{X}_{\varphi}$ be the set of $f\in C^1(M;\R)$ satisfying that there exist $n\in \N$ and an admissible word $a\in \mathbb{A}^{2n+1}$ such that
\begin{enumerate}
    \item $p\notin P(a)$;
    
    \item $\inf_{x \in P(a)} f(x) > \sup_{x \in \Lambda \backslash P(a)}f(x)$;
    
    \item $\inf_{x\in P(a)}|\nabla f(x)\cdot e^w(x)| > 0$;
\end{enumerate}

The next proposition is the main content of this section.
\begin{proposition}
\label{prop:genericityPotentials}
    The set $\mathcal{X}_{\varphi}$ is open and dense in $C^1(M;\R)$.
\end{proposition}
\begin{proof}
    Since the maps $f$, $\nabla f$ and $y \in \Lambda \mapsto e^w(y)$ are continuous, by compacity we see that the conditions 1-3 are $C^1$-open. Hence,  $\mathcal{X}_\varphi$ is open in $C^1(M;\mathbb{R})$. 

    It remains to show that $\mathcal{X}_\varphi$ is $C^1$-dense. For this purpose, fix $\varepsilon>0$ and consider $f \in C^1(M; \mathbb{R})$. Let $x_0 \in \Lambda$ be the maximal point of $f$ on $\Lambda$. Up to a small perturbation on $f$ we may assume that $x_0$ does not coincide with the fixed point $p\in \Lambda$.

    Now we divide the proof in two cases. In the first, we assume that $\nabla f(x_0)\cdot e^w(x_0) \neq 0$. Here, we can consider $n\in \N$ big enough and a finite admissible word $a\in \mathbb{A}^{2n+1}$ such that $p\notin P(a)$, $x_0\in P(a)$, and for every $x\in P(a)$,
    \begin{align} \label{Pot.Case1}
        \nabla f(x) \cdot e^w(x) \neq  0 \text{ and } f(x)> f(x_0)-\varepsilon/2.
    \end{align}
    Thus, using inequality \eqref{Pot.Case1} we can see that the function $g:\Lambda\to \R$ defined by
    \begin{align*}
        g(x) = \left\{
            \begin{array}{cc}
                f(x) + \varepsilon/2, & x\in P(a) \\
                f(x), & x\in \Lambda\backslash P(a),
            \end{array}
        \right.
    \end{align*}
    satisfies the conditions 2 and 3 in the definition of $\mathcal{X}_{\varphi}$ with $n$ and $P(a)$ given above. So, we can consider a partition of unity to extend $g$ to a $C^1$ function on $M$, which is $\varepsilon$-close to $f$ in the $C^1$ topology and $g\in \mathcal{X}_{\varphi}$. This concludes the $C^1$-density in this case.

    If $\nabla f(x_0)\cdot e^w(x_0) = 0$, consider first $n_0\in \N$ and a finite admissible word $b\in \mathbb{A}^{2n_0 + 1}$ such that the box $P(b)$ contains $x_0$, but does not contain $p$. Set,
    \begin{align}  \label{Pot.Case11}
        L: = \sup_{x \in P(b)}|x\cdot e^w(x_0)|,
        \quad
        N: = \sup_{x \in P(b)}\norm{\nabla f(x)}
        \quad
        \text{and}
        \quad
        0< \tau < \min\{\varepsilon, \varepsilon/8L\}.
    \end{align}
    Take $n\in \N$, $n\geq n_0$ and a finite admissible word $a\in \mathbb{A}^{2n+1}$ such that $P(a)\subset P(b)$, $x_0\in P(a)$ and for every $x\in P(a)$,
    \begin{align} \label{Pot.Case2}
        |\nabla f(x) \cdot e^w_{x_0}|
        <\dfrac{\tau}{4}, \;\; \norm{e^w(x)-e^w(x_0)}<\dfrac{\tau}{4(N+\tau)}   \text{ and } f(x)> f(x_0)-\dfrac{\varepsilon}{8}.
    \end{align}
    Define $g:\Lambda\to \R$ by
    \begin{align*}
        g(x) = \left\{
            \begin{array}{cc}
                f(x) + \varepsilon/2 + \tau\, x\cdot e^w(x_0), & x\in P(a) \\
                f(x), & x\in \Lambda\backslash P(a)
            \end{array}
        \right.,
    \end{align*}
    and as before extend $g$ to a $C^1$ real function from $M$. Notice that, by \eqref{Pot.Case11} and \eqref{Pot.Case2}, $g\in \mathcal{X}_{\varphi}$. Indeed, for every $x\in P(a)$,
    \begin{align*}
        g(x)
        = f(x)+\varepsilon/2 + \tau\, x\cdot e^w(x_0)
        > f(x_0) - \dfrac{\varepsilon}{8}+\dfrac{\varepsilon}{2}-\dfrac{\varepsilon}{8}
        = f(x_0) + \dfrac{\varepsilon}{4}
        > \sup_{y \in \Lambda\backslash P(a)} g(y).
    \end{align*}
    which proves condition 2. To see condition 3, notice that for every $x\in P(a)$,
    \begin{align*}
        \nabla g(x) \cdot e^w(x)
        & = (\nabla f(x) + \tau e^w(x_0))\cdot e^w(x)\\
        &= \nabla f(x) \cdot e^w(x_0) + \tau + ( \nabla f(x)+ \tau e^w(x_0) )\cdot( e^w(x) - e^w(x_0) ) \\
        & > -\dfrac{\tau}{4}+\tau - (\norm{\nabla f(x)}+ \tau) \norm{e^w(x) - e^w(x_0)}\\
        &>\dfrac{3}{4}\tau -(N+\tau)\dfrac{\tau}{4(N+\tau)}
        =\dfrac{\tau}{2}.
    \end{align*}
    Now we see that $g$ is a $\varepsilon$-perturbation of $f$ in the $C^1$-topology. But for every $x\in P(a)$,
    \begin{align*}
       & |f(x)-g(x)| 
       = |\varepsilon/2 + \tau\, x\cdot e^w(x_0)|
       < \varepsilon/2+\delta L
       < \varepsilon, \\
       & \norm{\nabla f(x)-\nabla g(x)}
       = \tau \norm{e^w(x_0)} = \tau
       < \varepsilon. 
    \end{align*}
    Therefore, $g\in \mathcal{X}_{\varphi}$ is $\varepsilon$-close to $f$ which conclude the proof of the density and the proposition.
\end{proof}

%%%%%%%%%%%%%%%%%%%%%%%%%%%%%%%%%%%%%%%%%%%%%%
\subsection{Lower bound the dimension of the Lagrange spectrum}
Assume that $\dim(M)\geq 3$. Let $\varphi\in \diff^r(M)$, $r>2$, be a diffeomorphism admitting a point $q_0$ of transversal homoclinic intersection associated with a $\varphi$-periodic point $p\in M$ that we assume to be fixed. Additionally, assume that $\varphi$ satisfy the conditions C1-C4 described in the Subsection \ref{subsec:choiceHorseshoe}.

Following the construction in Subsection \ref{subsec:choiceHorseshoe}, we can build a horseshoe $\Lambda_0$ for $\varphi^{N_0 + n_0}$ which is conjugated to $(\sigma,\, \{0,1\}^{\Z})$. Moreover, the horseshoe $\Lambda = \cup_{j=0}^{N_0+n_0 - 1} \varphi^{-j}(\Lambda_0)$ for $\varphi$ is conjugated a subshift of finite type $(\sigma,\, \Lambda)$ and satisfies the properties in Theorem~\ref{prop:niceHorseshoeProperties}. Denote the conjugacy map between $(\varphi,\, \Lambda)$ and $(\sigma,\, \Sigma)$ by $\Phi:\Lambda\to\Sigma$.  

Let $\mathcal{X}_{\varphi}$ be the subset constructed in Subsection \ref{section:genericPotentials} associated to the diffeomorphism $\varphi$ and the horseshoe $\Lambda$. Take $f\in \mathcal{X}_{\varphi}$. So, there exist $n\in \N$ and a $\Lambda$-admissible word $a\in \mathbb{A}^{2n+1}$ such that
\begin{enumerate}
    \item $p\notin P(a)$;
    
    \item $\inf_{x \in P(a)} f(x) > \sup_{x \in \Lambda \backslash P(a)}f(x)$;
    
    \item $\inf_{x\in P(a)}|\nabla f(x)\cdot e^w(x)| > 0$.
\end{enumerate}

Consider the set $\Lambda_{1/3}(a^c):= \Phi^{-1}(\Sigma_{1/3}(a^c))$.
Since $\Lambda$ is conjugated to a subshift of finite type and $\varphi^{N_0+n_0}|_{\Lambda_0}$ is conjugated to a full shift of two symbols $\sigma:\{0,1\}^{\Z}\to \{0,1\}^{\Z}$, taking $k:=N_0 + n_0$ such that $k(2n+1)$ is odd we can apply Proposition \ref{prop:trans_sigma2} guaranteeing that for $n$ sufficiently large, $\Lambda_{1/3}(a^c)$ is a transitive subhorseshoe of $\Lambda$ (notice that conditions 1-3 for $f$ holds if we instead of $a$ consider some other word $b\in \mathbb{A}^{2n'+1}$, $n'>n$, such that $C_0(b)$ contain the maximal point of $f$ on $\Lambda$ and is a subset of $C_0(a)$, so up to change $a$ we can assume that $n$ is as big as we want)

Also let $K_p(\varphi) := \pi_w^{-1}(K^w_l)\subset \Lambda_{1/3}(a^c)$, where $K^w_l$ is a regular Cantor set constructed in Subsection \ref{subsection:cantor} associated to a subhorseshoe $\Lambda_{1/3}(a^c)$ and satisfying the conclusion of Proposition \ref{prop:regularCantor}. 
\begin{proposition}
\label{prop:lowerBoundHDLagrange}
    It holds that
    \begin{align*}
        \dim_H(L(\Lambda, f)) \geq \dim_H(K_p(\varphi)) > 0.
    \end{align*}
\end{proposition}
\begin{proof}
    Define $g:= f\circ \Phi^{-1}$ and observe that by condition 2 for $f$, 
    \begin{align*}
        \sup_{\Sigma\backslash C_{-n}(a)}\, g < \inf_{C_{-n}(a)}\, g.
    \end{align*}
    Then, applying Proposition \ref{prop:150223.1}, there exists $j_0\in \N$ and an open subset $\mathcal{O}\subset \Sigma_{1/3}(a^c)$ such that,
    \begin{align}\label{eq:030323.1}
        f\circ \varphi^{j_0} \circ \Phi^{-1}\circ H_a(\mathcal{O})
        = g\circ \sigma^{j_0}\circ H_a(\mathcal{O})
        \subset L(\sigma,\, g)
        = L(\Lambda,f).
    \end{align}
    Moreover, for every $\theta\in \mathcal{O}$, we have that $\sigma^{j_0}\circ H_a(\theta) \in C_{-n}(a)$. Set,
    \begin{align*}
        \hat{\mathcal{O}}:= \Phi^{-1}(\mathcal{O})
        \quad
        \text{and}
        \quad
        \hat{H} := \Phi^{-1}\circ H_a \circ \Phi:\Lambda\to \Lambda.
    \end{align*}
  The subhorseshoe $\Lambda_{1/3}(a^c)$ of $\Lambda$ contains $\hat{\mathcal{O}}$ as an open subset and, by equation \eqref{eq:030323.1},
    \begin{align}\label{eq:030323.2}
        f\circ \varphi^{j_0}\circ \hat{H}(\hat{\mathcal{O}}) \subset L(\Lambda,\, f).
    \end{align}
    
    Take $x_0\in \hat{\mathcal{O}}$ and $\delta>0$. By transitivity of $\Lambda_{1/3}(a^c)$ we can find $y\in \Lambda_{1/3}(a^c)$ such that $d(y,p)<\delta$ and $d(\varphi^m(y), x_0)<\delta$, for some $m\in \N$. We can assume that $\delta$ is small enough such that
    \begin{enumerate}[label=(\roman*)]
        \item The brackets $z := [y,\, p]$, $z_0 = [x_0, \, \varphi^m(z)]$ are well defined;
         \item $W^s_{\delta}(z_0)\cap \Lambda_{1/3}(a^c) \subset \hat{\mathcal{O}}$;
        \item The unstable holonomies $H^u_{p,z}: W^s_{\delta}(p)\cap \Lambda_{1/3}(a^c)\to W^s_{\delta}(z)\cap \Lambda_{1/3}(a^c)$ and $H^u_{\varphi^m(z), z_0}: W^s_{\delta}(\varphi^m(z))\cap \Lambda_{1/3}(a^c)\to W^s_{\delta}(z_0)\cap \Lambda_{1/3}(a^c)$ are well defined.
    \end{enumerate}
    Consider the map $H^u_{p,z_0} =  H^u_{\varphi^m(z),z_0}\circ \varphi^m\circ H^u_{p,z}:W^s_{\delta}(p)\cap\Lambda_{1/3}(a^c)\to W^s_{\delta}(z_0)\cap\Lambda_{1/3}(a^c)$ and notice that by item \ref{prop:horseshoeProperties-item3} of Theorem~\ref{prop:niceHorseshoeProperties} this defines an $IC^{1+\alpha}$ map. Moreover, if we consider the regular Cantor set $K_{z_0} := H^u_{p,z_0}(K_p(\varphi)) \subset \hat{\mathcal{O}}$, then
    \begin{align*}
        IT_{z_0} K_{z_0} = \Delta H^u_{p,z_0}(p,p) \cdot IT_p K_p(\varphi) = \Delta H^u_{p,z_0}(p,p)\cdot E^w(p) = E^w(z_0).
    \end{align*}

    Using Proposition \ref{prop:boxMapHolonomy} and equation \ref{eq:090323.3} we see that there exists $k_0\in \N$ such that the map $\hat{H}$, restricted to $W^s_{\delta}(z_0)\cap \Lambda_{1/3}(a^c)$, can be written as
    \begin{align}\label{eq:030323.3}
        \hat{H} = \varphi^k\circ H^u_{z_0, \varphi^{-k}(\hat{H}(z_0))}: W^s_{\delta}(z_0)\cap \Lambda_{1/3}(a^c)\to W^s_{\delta_1}(\hat H(z_0))\cap \Lambda,
    \end{align}
    for some $\delta_1\in (0,\delta)$ and so, using again item \ref{prop:horseshoeProperties-item3} of Theorem \ref{prop:niceHorseshoeProperties}, we see that $\hat{H}$ is $IC^{1+\alpha}$. Hence, computing intrinsic derivatives we have,
    \begin{align*}
        \Delta (f\circ \varphi^{j_0}\circ \hat{H})(z_0,z_0)\cdot E^w(z_0)
        &= \nabla f((\varphi^{j_0}\circ \hat{H})(z_0))\cdot \Delta(\varphi^{j_0}\circ \hat H)(z_0,z_0)\cdot E^w(z_0)\\
        &= \nabla f((\varphi^{j_0}\circ \hat{H})(z_0))\cdot E^w(\varphi^{j_0}(\hat H(z_0)))\neq 0,
    \end{align*}
    where we used that $\varphi^{j_0}\circ \hat{H}(z_0)\in P(a)$ and condition 3 for $f$. Therefore, there is a small neighborhood of $z_0$ in $K_{z_0}$ such that the real function $f\circ\varphi^{j_0}\circ \hat{H}$ restricted to this neighborhood is a Lipschitz invertible function with Lipschitz inverse. Now, using equation \eqref{eq:030323.2} and the fact that $K_{z_0}$ has positive Hausdorff dimension (it is a regular Cantor set) we conclude that
    \begin{align*}
        \dim_H\, L(\Lambda, f) \geq \dim_H\left(
            f\circ \varphi^{j_0}\circ \hat H(K_{z_0})
        \right) = \dim_H(K_p(\varphi))> 0.
    \end{align*}
    This finishes the proof of the proposition.
\end{proof}

\begin{remark}
\label{remark:dimension2}
\normalfont
    In the case that $M$ is a compact surface, we can slight adapt the construction obtaining the same result as in Proposition \ref{prop:lowerBoundHDLagrange}. Indeed, it enough to observe that in this case, the stable Cantor set around $p$, $K_p(\varphi) := W^s_{\delta}(p)\cap \Lambda_{1/3}(a^c)$, is a Regular Cantor set, the holonomy maps are known to be differentiable and the map $\varphi\mapsto \dim_H(K_p(\varphi))$ varies continuously (see \cite{PaTa1995}). So, taking $E^w = E^s$ and repeating the arguments as before we obtain the analogous, in dimension two, of Proposition \ref{prop:lowerBoundHDLagrange}. Therefore, in what follows, we make use of this proposition without any restriction on the dimension. It is important to register that the result in dimension two was also obtained in \cite{IBMO2017}.
\end{remark}
\section{Proof of the results}
\label{section:proofResults}
Now we give proofs for the main results.

%%%%%%%%%%%%%%%%%%%%%%%%%%%%%%%%%%%%%%%%%%%%%
\subsection*{Proof of Theorem \ref{thm:Main}}
Consider now $\varphi\in \diff^{\infty}(M)$, admitting a point $q_0$ of transversal homoclinic intersection associated with a $\varphi$-periodic point $p\in M$. As before we assume without loss of generality that $p$ is fixed.

By Proposition \ref{prop:genericityConditionsC1-C4}, there exist $\varphi'\in \diff^{\infty}(M)$ $C^{\infty}$-close to $\varphi$ and a $C^1$-open neighborhood, $\mathcal{U}(\varphi')\subset \diff^{\infty}(M)$, of $\varphi'$ such that every $\tilde{\varphi}\in \mathcal{U}(\varphi')$ satisfies conditions C1-C4. Take $\tilde{\varphi}\in \mathcal{U}(\varphi')$. Then, by Proposition \ref{prop:genericityPotentials}, the set $\mathcal{X}_{\tilde{\varphi}}$ is open and dense in $C^1(M;\, \R)$. Therefore, for any $f\in \mathcal{X}_{\tilde{\varphi}}$, we may use Proposition \ref{prop:lowerBoundHDLagrange} and conclude that
\begin{align*}
    \dim_H(M(\tilde{\varphi},\, f)) \geq \dim_H(L(\tilde{\varphi},\, f)) \geq \dim_H(K_{\tilde{p}}(\tilde{\varphi})) > 0,
\end{align*}
where $\tilde{p}\in M$ is the hyperbolic continuation of $p$ for $\tilde{\varphi}$ and $K_{\tilde{p}}(\tilde{\varphi})\subset M$ is a regular Cantor set. This finishes the proof of Theorem \ref{thm:Main} (notice that the lower bound depends continuously on $\tilde{\varphi}\in \mathcal{U}(\varphi')$, see \cite{PaTa1995}).

\begin{remark}
\normalfont
    In the case that $M$ is a surface, it is enough to consider $\varphi\in \diff^2(M)$ in the above argument. The assumption that the diffeomorphism is $C^{\infty}$ in the higher dimensional case comes from the technical use of linearization to build the good horseshoe (see Subsection \ref{subsec:choiceHorseshoe}). Such maneuver is not necessary in dimension two, once any horseshoe $\Lambda$ for $\varphi$ has $C^{1+\gamma}$ holonomy maps and the stable Cantor set at any point is a regular Cantor set. So, the above argument provides a $C^1$-open neighborhood $\mathcal{U}(\varphi)\subset \diff^2(M)$ such that for any $\tilde{\varphi}\in \mathcal{U}(\varphi)$ we have (see Remark \ref{remark:dimension2}),
    \begin{align*}
        \dim_H(M(\tilde{\varphi},\, f)) \geq \dim_H(L(\tilde{\varphi},\, f)) \geq \dim_H(K_{\tilde{p}}(\tilde{\varphi})) > 0.
    \end{align*}
\end{remark}

%%%%%%%%%%%%%%%%%%%%%%%%%%%%%%%%%%%%%%%%%%%%%%%%%
\subsection*{Proof of Theorem \ref{thm:bigMarkov}}
Let $\varphi_0$ be a Morse-Smale diffeomorphism in ${\rm{Diff}}^\infty(M)$, and let $S_0$ be a fixed point sink for $\varphi_0$. Take $V$ an open neighborhood of $S_0$ such that $U:=\overline{V}$ is a compact subset of $W^s_{\text{loc}}(S_0)$ and $d(\varphi_0(U), V^c)), d(U, {\varphi_0^{-1}(V)}^c) > 2\varepsilon_1$, for some $\varepsilon_1>0$.

 Given a point $x_0$  in $V\setminus \varphi_0(U)$, consider an open neighborhood $W$ of $x_0$ such that $W \subset U\setminus \varphi_0(U)$ and $d(\overline{W}, V^c)), d(\overline{W}, {\varphi_0(U)}) > 2\varepsilon_2$, for some $\varepsilon_2>0$. Let $f_0: M \to \R$ be a  function,  satisfying:
$$f_0(x_0)>0, \; {\nabla f_0(x_0)}\neq 0 \; \text{ and } \; {\rm supp}(f_0)\subset W.$$

Since the set of Morse-Smale diffeomorphism is structurally stable, we can consider $\mathcal{V} \subset \diff^\infty(M)$ as a neighborhood of $\varphi_0$, satisfying that for every $\varphi \in \mathcal{V}$:
\begin{enumerate}[label=(\roman*)]
    \item $\varphi$ is a Morse-Smale diffeomorphism; 
    \item $d(\varphi(U), V^c)), d(U, {\varphi^{-1}(V)}^c) > \varepsilon_1 $;
     \item \label{final}$d(\overline{W}, {\varphi(U)}) > \varepsilon_2$.
\end{enumerate}

Let $\mathcal{W} \subset C^1(M; \R)$ be a neighborhood of $f_0$, such that for every $f \in \mathcal{W}$, we have:
\begin{equation} \label{msce1}
f(x_0)>f_0(x_0)/2, \;   \;{\nabla f(x_0)}\neq 0 {\rm \; and \; } |f(x)|<f_0(x_0)/4, {\rm \; for \; every \;} x \notin W.     
\end{equation}

Consider the set $\mathcal{U}:=\mathcal{V}\times \mathcal{W}$. Note that for every $(\varphi, f) \in \mathcal{U}$, $x \in W$ and $n$ a positive integer, by \ref{final},  we have $\varphi^n(x) \in \varphi(U)$ and $\varphi^{-n}(x) \notin U$. Since $\varphi(U)\cap W = \emptyset$ and $W \cap V^c=\emptyset$, using (\ref{msce1}), we get that $|f(\varphi^n(x))|,|f(\varphi^{-n}(x))|<f_0(x_0)/4$. Thus, for a small neighborhood $O$ of $x_0$, where $f{|_O}$ is bigger than $f_0(x_0)/4$, we have $f(O)\subset M(\varphi, f)$. Since $\nabla f(x_0) \neq 0$, the inverse function theorem guarantees that $f({O})$ contains an interval.

Therefore, for every $(\varphi,\, f) \in \mathcal{U}$, we have that  $M(\varphi,\, f)$ has non-empty interior but the Lagrange spectrum $L(\varphi,\, f) = f(Per(f))$ is finite.

%%%%%%%%%%%%%%%%%%%%%%%%%%%%%%%%%%%%%%%%%%%%%%%%%
\subsection*{Proof of Theorem \ref{thm:dichotomy}}
Assume first that $\dim M \geq 2$. Denote by $\mathcal{M}\mathcal{S}(M)$ the set of Morse-Smale diffeomorphisms in $\diff^{\infty}(M)$ and by $\mathcal{H}(M)$ the subset of $C^{\infty}$ diffeomorphisms which presents a transverse homoclinic intersection. Notice that $\mathcal{M}\mathcal{S}(M)$ and $\mathcal{H}(M)$ are $C^1$-open, disjoint subsets of $\diff^{\infty}(M)$ and the union of this sets, by \cite{Cr2010}, is $C^1$-dense in $\diff^{\infty}(M)$.

Let $\mathcal{C}(M)$ be the $C^1$-open and $C^{\infty}$-dense subset of $\mathcal{H}(M)$ given by Proposition \ref{prop:genericityConditionsC1-C4}. Define
\begin{align*}
    \mathcal{D}(M) := \left\{
        (\varphi,\, f)\in \diff^{\infty}(M)\times C^1(M;\, \R)\colon\,
        \varphi\in \mathcal{C}(M),\,
        f\in \mathcal{X}_{\varphi}
    \right\},
\end{align*}
and consider the data set
\begin{align*}
    \mathcal{R}(M) := \mathcal{D}(M)\cup (\mathcal{M}\mathcal{S}(M)\times C^1(M;\, \R)).
\end{align*}
Take $(\varphi, f) \in \mathcal{R}(M)$. If $(\varphi,\, f)\in \mathcal{M}\mathcal{S}(M)\times C^1(M;\, \R)$, then $L(\varphi,\, f)$ coincides with $f(\per(\varphi))$ which is finite by definition of Morse-Smale diffeomorphisms. If otherwise $(\varphi,\, f)\in \mathcal{D}(M)$, then by the proof of Theorem \ref{thm:Main}, we see that $L(\varphi,\, f)$ has positive Hausdorff dimension and since $\varphi \in \mathcal{C}(M)$ has an associated horseshoe, we have $h_{\text{top}}(\varphi)>0$. 

We claim that $\mathcal{R}(M)$ is $C^1$-open and $C^1$-dense in $\diff^{\infty}(M)$. First, we deal with the density. Take $(\varphi_0,f_0)\in \diff^{\infty}(M)\times C^1(M;\, \R)$. As mentioned before, we can find $\varphi_1\in \mathcal{M}\mathcal{S}(M)\cup \mathcal{H}(M)$, which is $C^1$-close to $\varphi_0$ (see \cite{Cr2010}). If $\varphi_1\in \mathcal{M}\mathcal{S}(M)$, then $(\varphi_1,\,f_0) \in \mathcal{R}(M)$. If otherwise $\varphi_1\in \mathcal{H}(M)$, then there exists $\varphi_2 \in \mathcal{C}(M)$ which is $C^{\infty}$-close to $\varphi_1$. By Proposition \ref{prop:genericityPotentials}, we can find $f_1\in \mathcal{X}_{\varphi_2}$ which is $C^1$-close to $f_0$. Then, the pair $(\varphi_2,\, f_1)\in \mathcal{D}(M)$ is $C^1$-close to $(\varphi_0,\, f_0)$ and this finishes the proof of the density.

Now we address the openness of $\mathcal{R}(M)$. By the fact that $\mathcal{M}\mathcal{S}(M)\times C^1(M;\, \R)$ is open it is enough to prove that $\mathcal{D}(M)$ is $C^1$-open. For that, take $(\varphi_0,\, f_0)\in \mathcal{D}(M)$.

Let $\delta_0>0$ such that if $\varphi\in \diff^{\infty}(M)$ and $\norm{\varphi - \varphi_0}_{C^1}<\delta_0$, then the construction in the Subsection \ref{subsec:choiceHorseshoe} provides a horseshoe $\Lambda_{\varphi}$ for $\varphi$. Decreasing $\delta_0$ if necessary, there exists a continuous map $\varphi\mapsto \Phi_{\varphi}$, where $\Phi_{\varphi}:\Sigma\to M$ is a conjugacy between the subshift of finite type $(\sigma,\, \Sigma)$ and the horseshoe $(\varphi,\, \Lambda_{\varphi})$. Moreover (See \cite[Theorem~8.3]{Sh2013}), there exists $C_0 = C_0(\delta_0)>0$ such that
\begin{align}\label{eq:070322.1}
    \sup_{\theta\in \Sigma}\, d(\Phi_{\varphi}(\theta), \Phi_{\varphi_0}(\theta)) \leq C_0\, \norm{\varphi - \varphi_0}_{C^1}.
\end{align}

Let $U$ be a small compact neighborhood of $\Lambda_{\varphi_0}$ and assume that $\delta_0$ is small enough such that $\Lambda_{\varphi}\subset U$. We take $U$ small such that for every $\varphi$, with $\norm{\varphi - \varphi_0}_{C^1}\leq \delta_0$, the map which associate for each $x\in \Lambda_{\varphi}$ the weak stable direction $E^w_{\varphi}(x)$ can be extended to $U$ satisfying that for every $x, y \in U$,
\begin{enumerate}
    \item There exists $C_1>0$ such that
    \begin{align}\label{eq:070322.2}
        d(E^w_{\varphi}(x),\, E^w_{\varphi_0}(x))
        < C_1\, \norm{\varphi - \varphi_0}_{C^1};
    \end{align}
    
    \item There exist uniform constants (independent of $\varphi$), $M>0$ and $\gamma_0\in (0,1)$ such that
    \begin{align}\label{eq:070322.3}
        d(E^w_{\varphi}(x),\, E^w_{\varphi}(y))
        \leq M\, d(x,\,y)^{\gamma_0}.
    \end{align}
\end{enumerate}
The first item can be found in \cite[Corollary~2.8]{CrPo2015}. For the second item, see \cite[Corollary~2.1]{BrPe1974} and \cite[Theorem 4.11]{CrPo2015}.

Since $f_0\in \mathcal{X}_{\varphi_0}$, we can find $\varepsilon>0$, $n\in \N$ and a finite $\Lambda_{\varphi_0}$-admissible word $a\in \mathbb{A}^{2n+1}$ such that
\begin{align}\label{eq:070322.4}
    \sup_{\Lambda_{\varphi_0}\backslash \Phi_{\varphi_0}(C_{-n}(a))}\, f_0 + \varepsilon 
    < \inf_{\Phi_{\varphi_0}(C_{-n}(a))}\, f_0 - \varepsilon
    \quad
    \text{and}
    \quad
    \inf_{x\in \Phi_{\varphi_0}(C_{-n}(a))}\,|\nabla f_0(x)\cdot e^w_{\varphi_0}(x)| \geq \varepsilon,
\end{align}
where $e^w_{\varphi_0}(x)$ denotes a unitary direction inside of $E^w_{\varphi_0}(x)$. Take $L :=\displaystyle \sup_{\norm{f - f_0}_{C^1}\leq \delta_0}\, \norm{\nabla f}_{\infty}$ and consider $\delta_1\in (0,\delta_0)$ such that $d(x,\, y)<\delta_1$ implies that
\begin{align}\label{eq:070322.5}
    \norm{\nabla f_0(x) - \nabla f_0(y)} < \varepsilon/8,
\end{align}
and consider 
\begin{align*}
    \delta  = \min\left\{
        \delta_1,\,
        \left(
            \frac{\varepsilon}{8ML}
        \right)^{1/\gamma_0}\, 1/C_0,\,
        \varepsilon/8,\,
        \frac{\varepsilon}{8LC_1},\,
        \frac{\varepsilon}{1 + L\, C_0}, \dfrac{\delta_1}{C_0}
    \right\}.
\end{align*}

Now consider $(\varphi,\, f)\in \diff^{\infty}(M)\times C^1(M;\, \R)$ such that
\begin{align}\label{eq:070322.6}
    \norm{\varphi - \varphi_0}_{C^1} <\delta
    \quad
    \text{and}
    \quad
    \norm{f - f_0}_{C^1} < \delta.
\end{align}
Then, using inequalities \eqref{eq:070322.1} and \eqref{eq:070322.6}, for every $\theta\in \Sigma$,
\begin{align}\label{eq:070322.7}
    \left|
        f(\Phi_{\varphi}(\theta)) - f_0(\Phi_{\varphi_0}(\theta))
    \right|
    &\leq \left|
        f(\Phi_{\varphi}(\theta)) - f(\Phi_{\varphi_0}(\theta))
    \right| + \left|
        f(\Phi_{\varphi_0}(\theta)) - f_0(\Phi_{\varphi_0}(\theta))
    \right|\nonumber\\
    &\leq \norm{\nabla f}\, d(\Phi_{\varphi}(\theta),\, \Phi_{\varphi_0}(\theta))
    + \norm{f - f_0}_{C^1}\\
    &\leq L\, C_0\, \delta + \delta
    \leq \varepsilon\nonumber
\end{align}
Furthermore, if we denote by $x := \Phi_{\varphi}(\theta)$ and $x_0 := \Phi_{\varphi_0}(\theta)$, then using triangular inequality and inequalities \eqref{eq:070322.1}, \eqref{eq:070322.2}, \eqref{eq:070322.3}, \eqref{eq:070322.6} and the choice of $\delta$ we have
\begin{align}\label{eq:070322.8}
    \left|
        \nabla f(x)\cdot e^w_{\varphi}(x)
    \right.
    &
    \left.
        - \nabla f_0(x_0)\cdot e^w_{\varphi_0}(x_0)
    \right|
    \leq \norm{\nabla f}_{\infty}\, d(e^w_{\varphi}(x),\, e^w_{\varphi}(x_0))
    + \norm{\nabla f(x) - \nabla f_0 (x)}\nonumber\\
    &+ \norm{\nabla f_0(x) - \nabla f_0(x_0)}
    + \norm{\nabla f_0}_{\infty}\, d(e^w_{\varphi}(x_0),\, e^w_{\varphi_0}(x_0))\nonumber\\
    &\leq L\, M\, d(x,\, x_0)^{\gamma_0}
    + \delta 
    + \varepsilon/8
    + L\, C_1\, \norm{\varphi-\varphi_0}_{C^1}\\
    &\leq L\, M\, (C_0\, \delta)^{\gamma_0} + \varepsilon/4 + L\, C_1\, \delta
    \leq \varepsilon/2\nonumber.
\end{align}

Now we are in the position to check the requirements needed to guarantee that $f\in \mathcal{X}_{\varphi}$. Indeed, using inequality \eqref{eq:070322.7}, we see that if $\theta\notin C_{-n}(a)$, then
\begin{align*}
    f(\Phi_{\varphi}(\theta))
    \leq f_0(\Phi_{\varphi_0}(\theta)) + \varepsilon
    \leq \sup_{\Lambda_{\varphi_0}\backslash \Phi_{\varphi_0}(C_{-n}(a))}\, f_0 + \varepsilon,
\end{align*}
and if $\theta\in C_{-n}(a)$,
\begin{align*}
    f(\Phi_{\varphi}(\theta))
    \geq f_0(\Phi_{\varphi_0}(\theta)) - \varepsilon
    \geq \inf_{\Phi_{\varphi_0}(C_{-n}(a))}\, f_0 - \varepsilon
\end{align*}
Then, by \eqref{eq:070322.4}, we have
\begin{align*}
    \sup_{\Lambda_{\varphi}\backslash \Phi_{\varphi}(C_{-n}(a))}\, f
    \leq \sup_{\Lambda_{\varphi_0}\backslash \Phi_{\varphi_0}(C_{-n}(a))}\, f_0 + \varepsilon
        < \inf_{\Phi_{\varphi_0}(C_{-n}(a))}\, f_0 - \varepsilon
    \leq \inf_{\Phi_{\varphi}(C_{-n}(a))}\, f.
\end{align*}
The second condition in the definition of $\mathcal{X}_{\varphi}$ can be checked similarly as follows: by \eqref{eq:070322.4} and \eqref{eq:070322.8}, for any $\theta\in C_{-n}(a)$,
\begin{align*}
    \left|
        \nabla f(\Phi_{\varphi}(\theta))\cdot e^w_{\varphi}(\Phi_{\varphi}(\theta))
    \right| \geq \left|
        \nabla f_0(\Phi_{\varphi_0}(\theta))\cdot e^w_{\varphi_0}(\Phi_{\varphi_0}(\theta))
    \right| - \varepsilon/2 \geq \varepsilon/2 > 0.
\end{align*}
Therefore $f\in \mathcal{X}_{\varphi}$ and so this shows that any pair $(\varphi, f)$ satisfying \eqref{eq:070322.6} belongs to $\mathcal{D}(M)$ proving that $\mathcal{D}(M)$ is $C_1$-open.

It remains to analyze the case where $\dim M = 1$. But, for circle diffeomorphisms it is known that Morse-Smale forms an open and dense class inside of $\diff^1_+(M)$ (see \cite[Section 1.15]{De2018}). Therefore, for any $f\in C^0(M;\, \R)$ we have that the Lagrange spectrum $L(\varphi,\, f) = f(Per(f))$ and it is finite. The concludes the proof of Theorem \ref{thm:dichotomy}.

\begin{remark}
    In dimension $2$
\end{remark}

%Bibliography
\bibliographystyle{abbrv}
\bibliography{bib.bib}

\information

\end{document}